\RequirePackage{silence}
\documentclass[a4paper]{article}
\usepackage[a4paper, total={6.5in, 8.66in}]{geometry}

\usepackage{graphicx}
\usepackage{xcolor}
\usepackage{enumitem}
\usepackage{amsmath}
\usepackage{amssymb}
\usepackage{url}
\usepackage[capitalise]{cleveref}
\usepackage{siunitx}
\usepackage{autonum} 
\usepackage{amsthm}
\usepackage{todonotes}
\usepackage{booktabs}
\usepackage{pifont}
\theoremstyle{plain} 
\newtheorem{theorem}{Theorem}
\newtheorem{proposition}{Proposition}
\newtheorem{lemma}{Lemma}
\theoremstyle{definition} 
\newtheorem{definition}{Definition}
\theoremstyle{remark}
\newtheorem{remark}{Remark}
\newcommand{\papertitle}{Hyperelastic constitutive model discovery with differentiable finite elements and structure-preserving neural networks}

\newcommand{\Id}{\mathbf{I}}
\newcommand{\Zr}{\mathbf{0}}
\newcommand{\ones}[1]{\mathbf{1}_{#1}}
\newcommand{\Rplus}{(0,+\infty)}
\newcommand{\RplusZ}{[0,+\infty)}
\newcommand{\dV}{\,\mathrm{d}V}
\newcommand{\dS}{\,\mathrm{d}\sigma}
\newcommand{\Cont}[1]{\mathcal{C}^{#1}}
\newcommand{\cof}{\operatorname{cof}}
\newcommand{\trace}{\operatorname{tr}}
\newcommand{\argmin}[1]{\underset{#1}{\operatorname{argmin}} }

\newcommand{\Var}{\operatorname{Var}}
\newcommand{\Exp}[1]{\mathbb{E}\left[#1\right]}

\newcommand{\Dim}{d}

\newcommand{\deformationMap}{\chi}
\newcommand{\x}{\mathbf{x}}
\newcommand{\displ}{\mathbf{d}}
\newcommand{\F}{\mathbf{F}}
\newcommand{\C}{\mathbf{C}}

\newcommand{\R}{\mathbf{R}}
\newcommand{\invI}{I_1}
\newcommand{\invII}{I_2}
\newcommand{\isoinvI}{\bar{I}_1}
\newcommand{\isoinvII}{\bar{I}_2}
\newcommand{\J}{J}
\newcommand{\W}{\mathcal{W}}
\newcommand{\Winv}{\mathcal{W}_{\mathrm{I}}}
\newcommand{\Wisoinv}{\mathcal{W}_{\bar{\mathrm{I}}}}
\newcommand{\Wpoly}{g}
\newcommand{\Wpnn}{\mathcal{W}_{\mathrm{PNN}}}
\newcommand{\Whnn}{\mathcal{W}_{\mathrm{HNN}}}
\newcommand{\Wbase}{\mathcal{W}_{\mathrm{base}}}
\newcommand{\Wgrow}{\mathcal{W}_{\mathrm{vol}}}
\newcommand{\Piola}{\mathbf{P}}

\newcommand{\loadBody}{\mathbf{f}}
\newcommand{\loadTraction}{\mathbf{h}}
\newcommand{\loadDispl}{\mathbf{g}}
\newcommand{\normal}{\mathbf{n}}
\newcommand{\BoundaryN}{\Gamma_{N}}
\newcommand{\BoundaryD}{\Gamma_{D}}
\newcommand{\reaction}[2]{R_{{\W}}^{#1#2}}

\newcommand{\spaceV}{V}
\newcommand{\spaceVZ}{V_0}
\newcommand{\spaceVh}{V_h}
\newcommand{\spaceVZh}{V_{h,0}}
\newcommand{\displTest}{\mathbf{v}}

\newcommand{\numOBS}{N_{\mathrm{exp}}}
\newcommand{\numOBSpts}[1]{N_{\mathrm{pts}}^{#1}}
\newcommand{\numOBSreaction}[1]{N_{\mathrm{reac}}^{#1}}
\newcommand{\displOBS}[2]{\mathbf{d}_{\mathrm{obs}}^{#1#2}}
\newcommand{\ptsOBS}[2]{\mathbf{x}^{#1#2}}
\newcommand{\reactionOBS}[2]{R_{\mathrm{obs}}^{#1#2}}

\newcommand{\stdNoise}{\sigma_{\mathrm{noise}}}

\newcommand{\SymmetryGroup}{\mathcal{G}}
\newcommand{\GLplus}{\mathrm{GL}^{+}(\Dim)}
\newcommand{\SO}{\mathrm{SO}(\Dim)}

\newcommand{\numDOF}{N_h}
\newcommand{\displDOF}{\underline{\mathbf{d}}}
\newcommand{\displDOFadjoint}{\underline{\mathbf{q}}}
\newcommand{\basisfun}[1]{\boldsymbol{\varphi}_{#1}}

\newcommand{\bilinearResidual}{r_{{\W}}}

\newcommand{\spaceW}{\mathcal{H}}

\newcommand{\numLay}{L}
\newcommand{\numNeur}[1]{n_{#1}}
\newcommand{\neur}{\mathbf{z}}
\newcommand{\neurOut}{z_{\mathrm{out}}}
\newcommand{\bias}[1]{\mathbf{b}_{#1}}
\newcommand{\weightOUT}{\mathbf{w}_{\mathrm{out}}}
\newcommand{\weightI}[1]{\mathbf{w}^1_{#1}}
\newcommand{\weightII}[1]{\mathbf{w}^2_{#1}}
\newcommand{\weightJ}[1]{\mathbf{w}^J_{#1}}
\newcommand{\weightZ}[1]{\mathbf{W}_{#1}}
\newcommand{\weightGrowth}{w_{\mathrm{vol}}}

\newcommand{\BaseWeightOUT}{\widetilde{\mathbf{w}}_{\mathrm{out}}}
\newcommand{\BaseWeightI}[1]{\widetilde{\mathbf{w}}^1_{#1}}
\newcommand{\BaseWeightII}[1]{\widetilde{\mathbf{w}}^2_{#1}}
\newcommand{\BaseWeightZ}[1]{\widetilde{\mathbf{W}}_{#1}}
\newcommand{\BaseWeightGrowth}{\widetilde{w}_{\mathrm{vol}}}

\newcommand{\Wconst}{\overline{\mathcal{W}}}
\newcommand{\act}{\rho}
\newcommand{\pos}{\tau}
\newcommand{\stressfactor}{\omega}
\newcommand{\stdInit}{\sigma_{\mathrm{init}}}
\newcommand{\NNparams}{\mathbf{w}}
\newcommand{\NNnumparams}{{N_w}}

\newcommand{\expW}{\overline{w}}

\newcommand{\DiscreteBilinearResidual}{\mathbf{r}}
\newcommand{\lossDispl}{\mathcal{L}}
\newcommand{\lossParam}{\widehat{\mathcal{L}}}

\newcommand{\figureSetupTwoD}{
\begin{figure}
    \vspace{-1.8cm}
    \includegraphics[trim=0 1.5cm 0 0, clip, page=1]{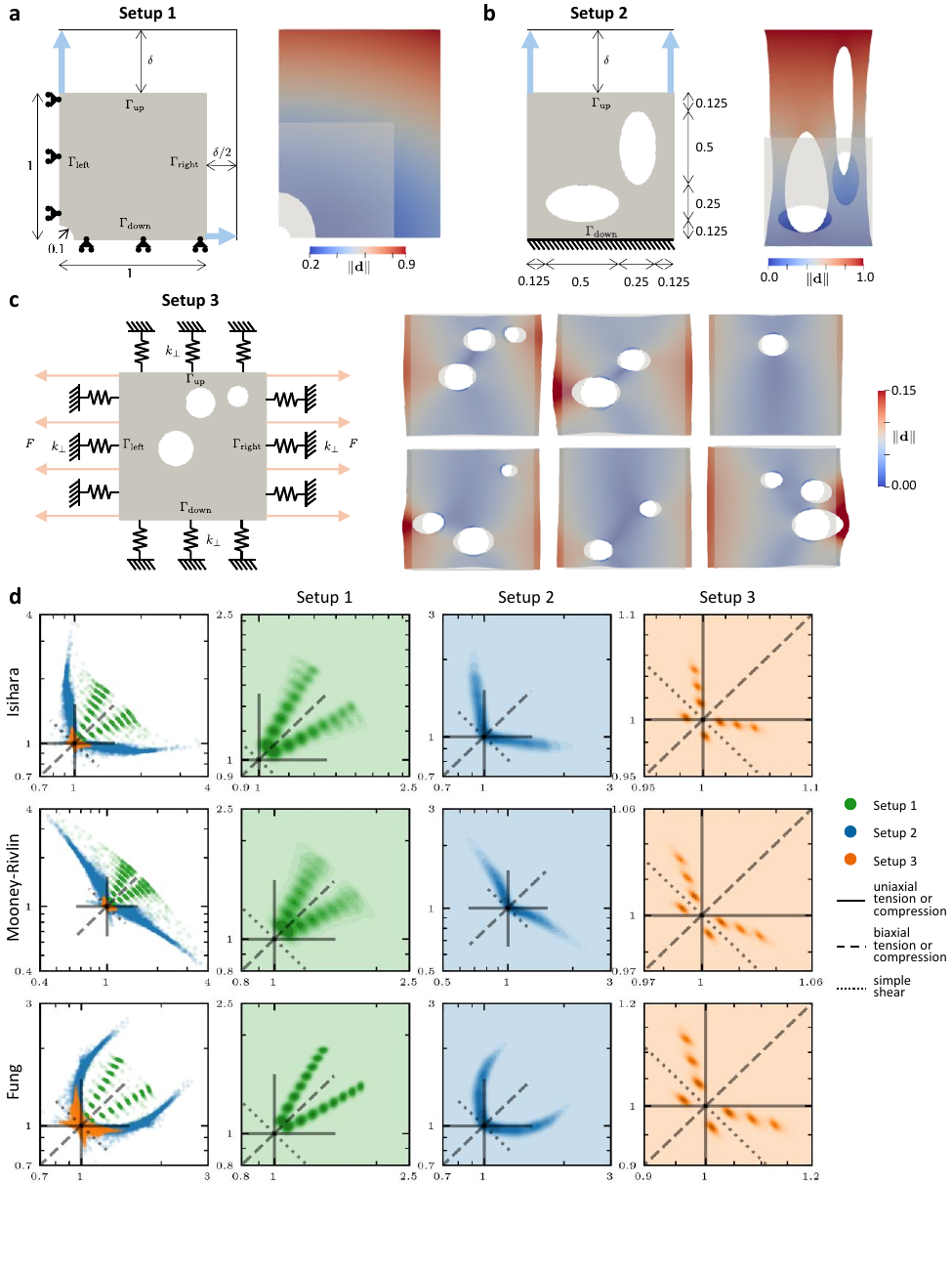}
    \vspace{-.7cm}
    \caption{
        \textbf{2D setups.} 
        (a)--(b)--(c): Graphical representation of Setups 1--2--3 (left), and corresponding sample solutions obtained with the IH model (right). 
        For Setup~3, six representative domains are shown.
        Light blue arrows represent displacement-controlled loads, whereas light orange arrows denote applied forces per unit surface.
        (d): Statistical distributions of the principal stretches {$(\lambda_1,\lambda_2)$} for the three setups considered. 
        The scatter plots in the first column show the relative location, in the principal-stretch plane, of the deformation states obtained for the three setups, while the remaining three columns detail the principal-stretch distributions for each setup, displayed via kernel density estimation.
        For clarity of visualization, the principal stretches are plotted without enforcing an ordering, i.e. treating the two stretches symmetrically.
        Black lines correspond to canonical deformations described in \ifappendix Appendix~\ref{sec:canonical_deformations}\else Supplementary Information\fi.
        }
    \label{fig:test_cases_2D}
\end{figure}
}

\newcommand{\figureSetupThreeD}{
\begin{figure}
    \includegraphics[trim=0 4.38cm 0 0, clip, page=2]{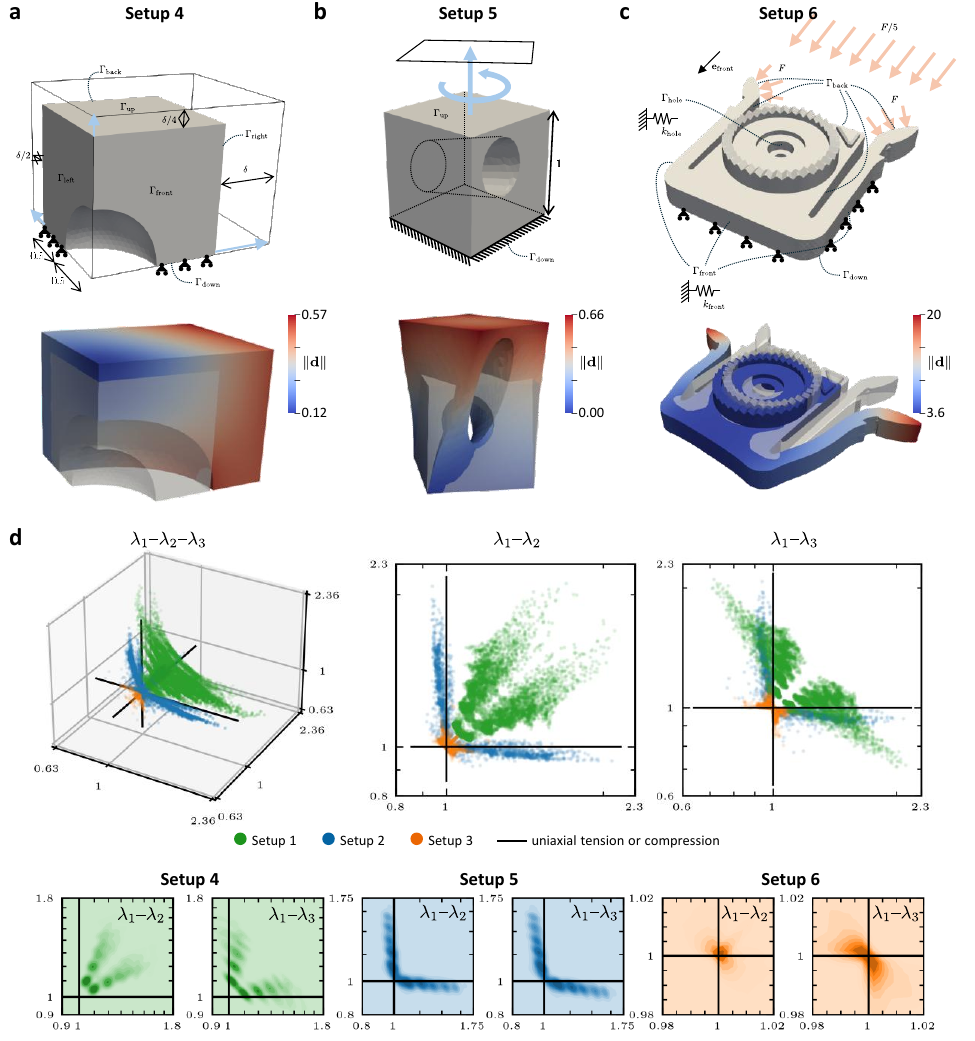}
    \caption{
        \textbf{3D setups.} 
        (a)--(b)--(c): Graphical representation of Setups 4--5--6 (top), and corresponding sample solutions obtained with the IH model (bottom). 
        Light blue arrows represent displacement-controlled loads, whereas light orange arrows denote applied forces per unit surface.
        (d): Statistical distributions of the principal stretches {$(\lambda_1,\lambda_2,\lambda_3)$} for the three setups considered. 
        The three scatter plots in the first row show the relative location, in principal-stretch space, of the deformation states obtained for the three setups, for the principal stretch triplet $(\lambda_1,\lambda_2,\lambda_3)$, and for the $(\lambda_1,\lambda_2)$ and $(\lambda_1,\lambda_3)$ pairs, respectively. 
        The second row shows the corresponding principal-stretch distributions for each setup, for the $(\lambda_1,\lambda_2)$ and $(\lambda_1,\lambda_3)$ pairs, displayed via kernel density estimation. 
        For clarity of visualization, the principal stretches are plotted without enforcing an ordering, i.e. treating the stretches symmetrically.
        Black lines correspond to the canonical uniaxial tension or compression deformations.
        }
    \label{fig:test_cases_3D}
\end{figure}
}

\title{\papertitle}
\author{Francesco Regazzoni$^{1}$}
\date{\footnotesize $^1$ MOX, Department of Mathematics, Politecnico di Milano, Milan, Italy}

\newif\ifappendix
\appendixtrue

\begin{document}
\maketitle

\begin{abstract}
The discovery of constitutive laws from experimentally accessible measurements is a central problem in nonlinear computational mechanics. Many data-driven constitutive identification approaches rely either on paired strain-stress data or on full-field displacement measurements, both of which are difficult to obtain in realistic three-dimensional settings. We present a differentiable finite element framework for the discovery of hyperelastic material laws from partial observations, including boundary-only displacement measurements and global reaction forces. The method embeds the nonlinear finite element equilibrium problem directly into the learning loop, so that candidate strain-energy densities are assessed through the deformation fields and reactions they induce. This formulation enforces mechanical equilibrium as a constraint and allows the loss function to be evaluated only at observed locations. To ensure physical admissibility and promote numerical solvability throughout training, the constitutive response is represented by Hyperelastic Neural Networks, a structure-preserving neural class that enforces residual energy and stress-free conditions, frame indifference, isotropic material symmetry, polyconvexity, coercivity, and controlled volumetric growth by construction. The resulting PDE-constrained learning problem is solved using a quasi-Newton strategy combined with continuation and solver-aware backtracking. Numerical experiments in two- and three-dimensional finite elasticity demonstrate accurate recovery of hyperelastic isotropic responses from boundary-only data, robustness to measurement noise, and generalization across geometries, loading conditions, and boundary conditions.
\end{abstract}

\section{Introduction}
\label{sec:intro}

Constitutive modeling is a central component of nonlinear solid mechanics. 
In hyperelasticity, the mechanical response of a material is encoded by a strain-energy density, from which the stress field and the equilibrium configuration of a body follow \cite{antman2005nonlinear,gurtin2010mechanics}. 
In standard practice, the functional form of this energy is selected \emph{a priori} from a family of constitutive models and only a limited number of material parameters are calibrated from experiments. 
While effective in many engineering applications, this approach can be restrictive when the true material response is complex, when the relevant deformation modes are not known in advance, or when the available tests do not provide sufficient information to justify a prescribed model structure.
More recently, a new data-driven paradigm has emerged that relaxes the need to prescribe a closed-form constitutive family in advance, and instead infers the material response from data \cite{kirchdoerfer2016data,mahmoudabadbozchelou2024unbiased,holthusen2025generalized}.

This shift aligns with a broader movement toward the automated discovery of governing physical laws from data \cite{karniadakis2021physics,champion2019data,zhou2018learning,cranmer2020discovering,iten2020discovering}, spanning approaches such as sparse identification of nonlinear dynamics \cite{brunton2016discovering,champion2019data}, symbolic regression \cite{udrescu2020ai}, and neural networks \cite{raissi2020hidden,cranmer2020lagrangian}.
In solid mechanics, constitutive law identification can be viewed as a particular instance of this more general scientific discovery problem, where the objective is to infer the local material behavior underlying observed macroscopic responses.

Within this paradigm, a distinction is commonly made between \emph{model-free} and \emph{model discovery} approaches for mechanics.
Model-free methods \cite{kirchdoerfer2016data,conti2018data,kirchdoerfer2018data,stainier2019model} do not yield an explicit constitutive relationship; instead, a collection of recorded strain--stress pairs is used to infer the material response online, by locally approximating the behavior based on data points closest to the current state of each material point.
While effective in certain settings, these approaches do not explicitly identify a reusable constitutive model and are typically both data-hungry and memory-intensive, as the full dataset must be queried during the simulation.

In contrast, model discovery methods aim to identify an explicit constitutive model in the classical sense, yet without prescribing its functional form \emph{a priori}.
Such models may be based on symbolic regression techniques \cite{flaschel2021unsupervised,linka2021constitutive,flaschel2022discovering,flaschel2023automated,flaschel2023automatedbrain}, where the constitutive law is discovered by combining elements from a library of simple mathematical expressions, or on machine learning models such as neural networks \cite{shen2004neural,liang2008neural,huang2020learning,thakolkaran2022nn}, Gaussian processes \cite{frankel2020tensor}, or neural ordinary differential equations \cite{tac2022data}.
Symbolic approaches offer enhanced interpretability, whereas machine learning models usually provide superior expressive power.
Moreover, several architectures \cite{vlassis2020geometric,linka2021constitutive,fernandez2021anisotropic,tac2022data,klein2022polyconvex,thakolkaran2022nn} have been proposed to automatically embed physical and mathematical constraints into the learned model, thereby restricting the hypothesis space \emph{a priori} (for a comprehensive overview in this respect, see \cite{linden2023neural}).

Regardless of whether the model is symbolic or black-box, existing methods also differ in the type of data they require.
A large part of the constitutive-learning literature assumes access to local strain--stress data \cite{linka2023new,tac2022data}.
However, such data are often unrealistic to obtain in practice: stress measurements are typically available only in simple mechanical tests, such as uniaxial tension, which fail to capture the full complexity of constitutive behavior.
Indeed, the simultaneous measurement of all tensorial components of strain and stress is practically infeasible.

A more realistic experimental setting involves direct measurements of the displacement field of a specimen, for instance obtained via digital image correlation (DIC) techniques \cite{gdoutos2022DIC}, together with global force measurements recorded during the deformation process.
The Virtual Field Method (VFM)~\cite{grediac2006virtual} has been leveraged for addressing material discovery problems in this setting \cite{huang2020learning}, e.g. by the EUCLID method \cite{flaschel2021unsupervised,flaschel2022discovering}.
The VFM identifies the constitutive model that minimizes the residual of the force balance, given the measured displacement field.
The methods belonging to this class have enabled significant progress in stress-label-free constitutive identification and material model discovery.
Still, because the measured displacement field enters the weak form directly, these approaches typically require sufficiently rich displacement information over the region where the equilibrium residual is evaluated.
Moreover, strains must be obtained from measured displacements, which can make the identification sensitive to measurement noise and may require smoothing or denoising procedures.
This requirement may become restrictive when observations are incomplete, defective, or limited to a subset of the boundary, as may occur in three-dimensional experiments where internal displacement measurements are not available.
A related full-field-data requirement appears in ADiMU, a differentiable framework for history-dependent constitutive laws in which global model discovery relies on an equilibrium force residual evaluated from the measured full-field displacement history \cite{ferreira2025automatically}.
{Material Fingerprinting \cite{flaschel2026material,flaschel2026unsupervised} has shown that constitutive models can be identified from a limited set of surface displacement measurements and global reaction forces, although through a fundamentally different lookup-table strategy based on a precomputed database for a prescribed experimental setup.}

In this work, in order to relax the full-field displacement requirement, we address {hyperelastic} material model discovery by embedding a differentiable finite element solver for the equilibrium equations directly into the training process, thereby establishing a direct link between candidate constitutive laws and the resulting deformation fields.
In the considered approach, which we denote by Neural Differentiable Finite Element Method (Neural-DFEM), the loss function is formulated in terms of the discrepancy between available measurements and model predictions on the training specimens.
Estimating material model parameters within a PDE-constrained setting, as a matter of fact, is a well-established approach in the community, known as finite element model updating (FEMU) \cite{avril2008overview,chen2025finite,kumar2025comparative}.
FEMU has been widely adopted in a variety of settings, thanks to its flexibility \cite{avril2008overview}.
Here, we leverage this PDE-constrained setting in combination with a neural-network-based constitutive ansatz, using automatic differentiation through the neural network together with adjoint-based differentiation through the finite element equilibrium problem to efficiently compute gradients and train the model.
After the initial arXiv posting of the present work, a closely related approach, developed in parallel, was presented in \cite{knipper2026finite}, which combines neural constitutive models with a differentiable finite element framework within a FEMU-type identification setting, focusing on end-to-end automatic differentiation and identification from experimental deformation data.

However, embedding a differentiable solver within the training loop introduces two main challenges. 
The first challenge arises from the requirement that the elastostatic equilibrium problem associated with each intermediate candidate model remain solvable throughout the optimization process; failure of the finite element solver at any stage would lead in fact to training breakdown.
To this end, we rely on a neural network ansatz, termed Hyperelastic Neural Network (HNN), which provides mathematically provable guarantees of physical and thermodynamic consistency, as well as sufficient conditions for {the existence of equilibria}.

The second challenge concerns computational cost and robustness with respect to initialization.
This issue is addressed through the use of a quasi-Newton optimization strategy, which significantly reduces the number of training epochs required for convergence, combined with a continuation strategy during the initial stages of training. 
In addition, to mitigate the local convergence properties of the Newton solver employed for equilibrium, a backtracking strategy is incorporated into the line search of the quasi-Newton training algorithm, automatically rejecting parameter updates that would move the system outside the Newton basin of attraction.
Finally, we propose a carefully designed initialization strategy for the neural network weights.
These numerical strategies together promote robustness and computational efficiency of Neural-DFEM.

The contributions of this work are the following:
\begin{itemize}
    \item We formulate a practical equilibrium-constrained learning procedure for neural constitutive model discovery {for hyperelastic materials} from partial mechanical observations. This is achieved by combining automatic differentiation, adjoint-based gradient computation, quasi-Newton optimization, continuation, and solver-aware backtracking to obtain stable solver-in-the-loop training.
    \item We introduce a structure-preserving neural network constitutive class that enforces the fundamental requirements of a hyperelastic law. 
    Several of these requirements are by now well established in the literature on physics-constrained neural constitutive modeling \cite{linden2023neural}; here, we additionally incorporate conditions ensuring {the existence of equilibria}, which is crucial for the numerical solvability across the whole training trajectory. 
    In particular, we prove a M{\"u}ller-type coercivity, which together with the other properties guarantees the existence of solutions to the equilibrium problem.
    \item We show, on two- and three-dimensional finite-elasticity benchmarks, that the resulting framework can learn hyperelastic responses from incomplete observations, including boundary-only displacement data and global reaction forces, and can generalize across geometries, loading paths, and boundary conditions.
\end{itemize}

The numerical results are intended as a controlled validation of the proposed framework under known ground-truth material laws and prescribed observation operators.
They show that differentiable finite element solvers, coupled with physically constrained neural constitutive models, provide a viable route to hyperelastic model discovery under partial observability.
An overview of the Neural-DFEM framework is shown in Fig.~\ref{fig:graphical_abstract}.

\begin{figure}
    \includegraphics[trim=0 5.65cm 0 0, clip, page=6]{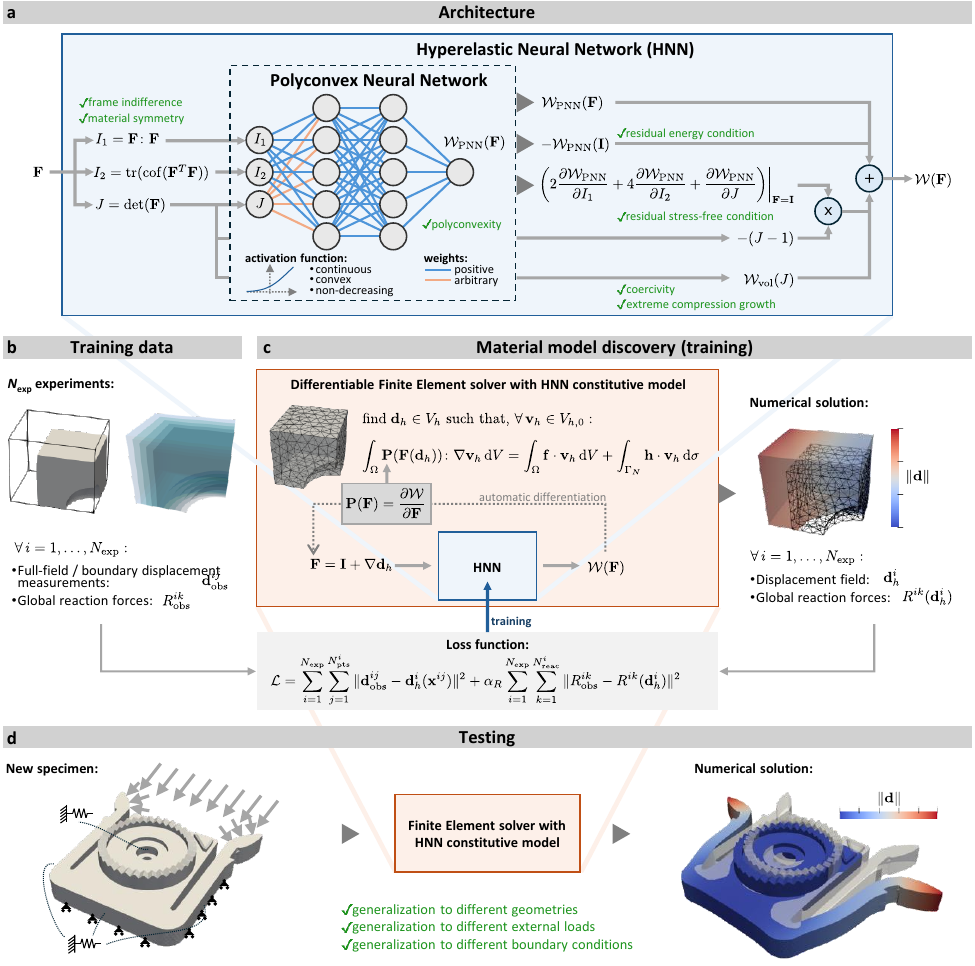}
    \caption{
        \textbf{Overview of the Neural-DFEM framework.} 
        (a) HNNs constitute a family of neural networks designed to approximate the strain energy density functional of a hyperelastic isotropic material. By construction, an HNN satisfies seven physical and mathematical requirements that ensure {existence of equilibria} and physical consistency, as indicated in the figure alongside the architectural components that enforce them.
        The figure highlights the main design traits of HNNs, including weights with controlled sign, and a specific choice of the activation function. Optionally, HNNs can include skip connections, although they are not represented in the figure.
        (b) The training data consist of (possibly partial) displacement measurements and global reaction forces, coming from $\numOBS$ experiments.
        (c) Training is performed by embedding a differentiable finite element solver within the loss evaluation. The loss measures the discrepancy between predicted and observed data, while the Piola-Kirchhoff stress tensor is obtained by differentiating the HNN output with respect to the deformation gradient.
        (d) Once trained, the HNN can be integrated into standard finite element solvers to predict the deformation of new specimens, enabling generalization across geometries, loading conditions, and boundary conditions.
        }
    \label{fig:graphical_abstract}
\end{figure}

\section{Methods}
\label{sec:methods}

\subsection{Notation} 
\label{sec:notation}

We denote with capital bold letters matrices, e.g., $\mathbf{A}, \mathbf{B} \in \mathbb{R}^{n \times n}$, and by lowercase bold letters vectors, e.g., $\mathbf{a}, \mathbf{b} \in \mathbb{R}^{n}$. 
The Frobenius norm is denoted by $|\cdot|$, that is $| \mathbf{A} | = \trace(\mathbf{A}^T \mathbf{A})^{1/2} = (\sum_{i,j} A_{ij}^2)^{1/2}$.
With $\mathbf{a} \succeq b$ (respectively, $\mathbf{a} \succ b$), we mean that each component of $\mathbf{a}$ is greater or equal to (respectively, greater than) the corresponding component of $\mathbf{b}$.
We denote by $\ones{n} \in \mathbb{R}^n$ the $n$-dimensional vector with all entries equal to 1, and by $\ones{n \times m} \in \mathbb{R}^{n\times m}$ the $n \times m$ matrix with all entries equal to 1.

We denote by $\Dim = 3$ the space dimension.
$\Id \in \mathbb{R}^{\Dim \times \Dim}$ is the identity tensor, $\Zr \in \mathbb{R}^{\Dim \times \Dim}$ is the zero tensor.
The set of second order tensors with positive determinant is denoted as 
$\GLplus = \{ \mathbf{A} \in \mathbb{R}^{\Dim \times \Dim} \mid \det \mathbf{A} > 0 \}$.
The special orthogonal group is denoted as
$\SO = \{ \mathbf{A} \in \mathbb{R}^{\Dim \times \Dim} \mid \mathbf{A}^T\mathbf{A} = \Id, \det \mathbf{A} = 1 \}$.


\subsection{Material model discovery: problem setting}
\label{sec:problem_setting}

Let us consider a bounded domain $\Omega \subset \mathbb{R}^{3}$, representing the stress-free reference configuration of an elastic body, and $\deformationMap\colon\Omega \to \mathbb{R}^{3}$ the deformation map such that $\deformationMap(\Omega)$ is the deformed configuration. 
We denote by $\displ(\x) = \deformationMap(\x) - \x$ the displacement field, by $\F = \nabla\deformationMap = \Id + \nabla\displ$ the deformation gradient tensor, and by $\J = \det\F$ its determinant.
We denote by $\C = \F^T\F$ the right Cauchy-Green tensor, and its invariants by $\invI = \trace\C$ and $\invII = \frac{1}{2}(\trace(\C)^2 - \trace(\C^2))$.
Moreover, we denote the isochoric invariants as $\isoinvI  = \J^{-2/3} \invI $ and $\isoinvII = \J^{-4/3} \invII$.

We consider a hyperelastic material, characterized by a strain energy density functional $\W \colon \GLplus \to \mathbb{R} \cup \{+\infty\}$, such that the Piola-Kirchhoff stress tensor is given by $\Piola(\F) = \frac{\partial\W}{\partial\F}$.
In such a way, the strain energy density $\W$ characterizes the material constitutive law. 
The knowledge of $\W$ allows to predict the material response to deformations. The goal of this work is to reverse this process, that is to discover the function $\W$ by observing the deformation of specimens of the considered materials.

The following differential problem models the configuration assumed by a body when subject to a distributed load $\loadBody$, to a boundary load $\loadTraction$ on $\BoundaryN \subset \partial\Omega$ and assigned displacement $\loadDispl$ on $\BoundaryD \subset \partial\Omega$ (where $\BoundaryN$ and $\BoundaryD$ form a disjoint partition of $\partial\Omega$):
\begin{equation} \label{eqn:equilibrium_strong_form}
    \left\{
    \begin{aligned}
        - \nabla \cdot \Piola(\F) &= \loadBody && \text{in } \Omega,\\
        \Piola(\F) \normal &= \loadTraction && \text{on } \BoundaryN,\\
        \displ &= \loadDispl && \text{on } \BoundaryD,\\
    \end{aligned}
    \right.
\end{equation}
where $\normal$ denotes the outward normal vector to $\BoundaryN$.
Notice that method proposed in this work extends to different kind of boundary conditions (such as follower loads, spring-like boundary conditions, sliding conditions, symmetric conditions), but, for the sake of brevity, we will carry on the presentation restricting ourselves to the case of \eqref{eqn:equilibrium_strong_form}.

We assume that $\numOBS$ experiments are conducted, each corresponding to a potentially different specimen geometry $\Omega^i$, boundary conditions type, and loads $\loadBody^i$, $\loadTraction^i$, $\loadDispl^i$. 
Let $\displ^i\colon\Omega_i \to \mathbb{R}^{3}$ denote the displacement field resulting from the $i$-th experiment, for $i = 1,\dots,\numOBS$.
For each experiment, we dispose of measurements of the displacement at a collection of $\numOBSpts{i}$ points. 
Specifically, we denote by $\displOBS{i}{j} \simeq \displ^i(\ptsOBS{i}{j})$ a (typically noisy) measurement of the displacement of the $i$-th experiment at the $j$-th point (denoted by $\ptsOBS{i}{j} \in \Omega^i$), for $j = 1,\dots,\numOBSpts{i}$.

Furthermore, typical experimental setups allow to measure the global reaction forces acting on the boundaries where prescribed displacements are assigned, that is in correspondence of Dirichlet boundary conditions. 
Suppose that the $i$-th experiment has $\numOBSreaction{i}$ Dirichlet boundary segments, denoted by $\BoundaryD^{ik} \subset \partial\Omega^i$, for $k=1,\dots,\numOBSreaction{i}$. 
Then, we denote by $\reaction{i}{k}$ the function that maps a displacement field into the associated normal reaction force{, where the subscript $\W$ stresses the dependence on the strain energy density function at hand}:
\begin{equation}
    \reaction{i}{k}(\displ) := \int_{\BoundaryD^{ik}} \Piola(\Id + \nabla \displ) \normal \cdot \normal \dS.
\end{equation}
Using this notation, we denote by $\reactionOBS{i}{k} \simeq \reaction{i}{k}(\displ^i)$ the (typically noisy) measured reaction force.

To summarize, the full set of training data is given by:
\begin{equation}
    \{ \displOBS{i}{j} \}_{i = 1,\dots,\numOBS}^{j = 1,\dots,\numOBSpts{i}}
    , \qquad
    \{ \reactionOBS{i}{k} \}_{i = 1,\dots,\numOBS}^{k = 1,\dots,\numOBSreaction{i}}.
\end{equation}

\subsubsection{Finite element approximation} \label{sec:FEM_formulation}

Considering problem~\eqref{eqn:equilibrium_strong_form}, let us introduce the space of kinematically admissible displacements $\spaceV$, and the space of admissible variations (test function space) $\spaceVZ$:
\begin{equation}
    \begin{split}
            \spaceV  &:= \{\displ \in [H^1(\Omega)]^3 \text{ s.t. } \displ = \loadDispl \text{ on } \BoundaryD\}, \\
            \spaceVZ &:= \{\displ \in [H^1(\Omega)]^3 \text{ s.t. } \displ = \mathbf{0} \text{ on } \BoundaryD\}.
    \end{split}
\end{equation}
The equilibrium residual $\bilinearResidual \colon \spaceV \times \spaceVZ \to \mathbb{R}${, where we stress the dependence on $\W$ through the subscript,} is given by:
\begin{equation}
    \bilinearResidual(\displ,\displTest) = \int_{\Omega} \Piola(\F) \colon \nabla \displTest \dV
    - \int_{\Omega} \loadBody \cdot \displTest \dV
    - \int_{\BoundaryN} \loadTraction \cdot \displTest \dS.
\end{equation}
The weak formulation of problem~\eqref{eqn:equilibrium_strong_form} reads as finding $\displ \in \spaceV$ such that $\bilinearResidual(\displ,\displTest) = 0$ for all $\displTest \in \spaceVZ$.

Let us consider a finite element space $\spaceVh \subset \spaceV$, defined on a triangulation of the domain $\Omega$, and let us define $\spaceVZh = \spaceVh \cap \spaceVZ$. 
Let $\{ \basisfun{n} \}_{n=1}^{\numDOF}$ be a basis for $\spaceVZh$.
Then, each function of $\spaceVh$ can be expressed as $\displ_h(\mathbf{x}) = \sum_{n=1}^{\numDOF} d_i \basisfun{n}(\mathbf{x}) + \loadDispl_h$, where $\loadDispl_h \in \spaceVh$ is the lifting of the boundary datum $\loadDispl$.
By using this notation, the finite element approximation of problem~\eqref{eqn:equilibrium_strong_form} consists of finding $\displ_h \in \spaceVh$ such that $\bilinearResidual(\displ_h,\displTest_h) = 0$ for all $\displTest_h \in \spaceVZh$, or, equivalently, such that $\bilinearResidual(\displ_h,\basisfun{n}) = 0$ for $n = 1, \dots, \numDOF$.
In what follows, we will employ the superscript $i$ to denote the variables associated with the $i$-th experiment (e.g., $\spaceVh^i$, $\displ_h^i$, $\numDOF^i$, $\basisfun{n}^i$).


\subsubsection{Material law discovery}

We introduce a hypothesis space $\spaceW \subset \{ \W \colon \GLplus \to \mathbb{R} \cup \{+\infty\} \}$, that is a set of candidate material laws.
The issue of designing a convenient hypothesis space, sufficiently rich but at the same time numerically tractable, will be addressed below.
For the moment, let us assume this space to be given.
Constitutive model discovery consists in finding the material law $\W \in \spaceW$ that is, in a suitable sense, most consistent with the observations $\displOBS{i}{j}$ and $\reactionOBS{i}{k}$. 
Such material discovery problem can be formulated in different manners.

A first possibility is to rely on the VFM formulation (such in \cite{huang2020learning,flaschel2021unsupervised}), which consists in looking for the material law such that the equilibrium residual is minimized:
\begin{equation}\label{eqn:discovery_VFM}
    \W^* = \argmin{\W \in \spaceW} \left(
      \sum_{i=1}^{\numOBS} \sum_{n=1}^{\numDOF^i}          \| \bilinearResidual(\displ_{h,\text{obs}}^i,\basisfun{n}^i) \|^2
    + \alpha_R           \sum_{i=1}^{\numOBS} \sum_{k=1}^{\numOBSreaction{i}} \| \reactionOBS{i}{k} - \reaction{i}{k}(\displ_{h,\text{obs}}^i) \|^2
    \right),
\end{equation}
where $\displ_{h,\text{obs}}^i \in \spaceVh^i$ is a suitably denoised projection of the observed displacement $\{ \displOBS{i}{j} \}_{j = 1,\dots,\numOBSpts{i}}$ onto the finite element space $\spaceVh^i$, and where $\alpha_R > 0$ is a weight constant.
Clearly, this approach is valid only when the observed displacement $\{ \displOBS{i}{j} \}_{j = 1,\dots,\numOBSpts{i}}$ covers the whole domain $\Omega^i$, otherwise the equilibrium residual cannot be computed.
{Additional regularization or sparsity-promoting terms can be added to \eqref{eqn:discovery_VFM} (see e.g. \cite{alheit2026cann}}).

A second alternative consists in changing the role of the equilibrium condition from a penalization term to a constraint, and minimizing the mismatch between the predicted and the observed displacement.
The resulting formulation reads:
\begin{equation}\label{eqn:discovery_DFEM}
    \W^* = \argmin{\W \in \spaceW} \left(
      \sum_{i=1}^{\numOBS} \sum_{j=1}^{\numOBSpts{i}}      \| \displOBS{i}{j}    - \displ_h^{i,{\W}}(\ptsOBS{i}{j}) \|^2
    + \alpha_R      \sum_{i=1}^{\numOBS} \sum_{k=1}^{\numOBSreaction{i}} \| \reactionOBS{i}{k} - \reaction{i}{k}(\displ_h^{i,{\W}}) \|^2
    \right),
\end{equation}
subject to the constraints $\bilinearResidual(\displ_h^{i,{\W}},\displTest_h^i) = 0$ for all $\displTest_h^i \in \spaceVZh^i$.
This is the approach followed in this work.
Specifically, the normalization constant is here defined as
\begin{equation}
    \alpha_R = \frac{
        \sum_{i=1}^{\numOBS} \sum_{j=1}^{\numOBSpts{i}}      \| \displOBS{i}{j}    \|^2
    }{
        \sum_{i=1}^{\numOBS} \sum_{k=1}^{\numOBSreaction{i}} \| \reactionOBS{i}{k} \|^2
    }
\end{equation}
so as to balance the relative contributions of the displacement and reaction terms in the objective function.

A third alternative is to treat both the equilibrium condition and the displacement observations as penalization terms.
In this case, both the constitutive law $\W$ and the displacement fields $\displ^i$ are variables of the optimization problem.
This formulation, when neural networks are used as ansatz for the displacement fields and the equilibrium condition is expressed in strong form, takes the form of a Physics-Informed Neural Network (PINN) \cite{raissi2019physics}.
As preliminary investigations with this technique highlighted very little stability of the training process and serious difficulties in convergence, we decided not to proceed further, and focus on the second approach only.


\subsection{Designing the hypothesis space}

In this section we address the issue of designing a hypothesis space $\spaceW$ of candidate constitutive laws that is sufficiently rich to encompass a wide range of possible materials, while being computationally tractable.

\subsubsection{Constitutive modeling requirements} 
\label{sec:W_requirements}

To define a suitable hypothesis space $\spaceW$, let us first review the requirements that the strain energy density functional $\W \colon \GLplus \to \mathbb{R} \cup \{+\infty\}$ must satisfy for physical meaningfulness and {for existence of solutions to the equation} \eqref{eqn:equilibrium_strong_form} \cite{truesdell2004non,antman2005nonlinear,gurtin2010mechanics}.

\begin{enumerate}[label=R\arabic*]
    \item \textbf{Residual energy condition:} \label{req:residual-energy}
    $\W(\Id) = 0$.
    
    This condition fixes the zero of the energy and ensures uniqueness of the energy reference. 
    Although the energy always appears under differentiation in the stress, having a reference value is convenient for consistency.

    \item \textbf{Residual stress free condition:} \label{req:residual-stress-free}
    $\Piola(\Id) = \Zr$.

    This ensures that, in the rest configuration ($\displ = \mathbf{0}$), no residual stress is present.    

    \item \textbf{Frame indifference:} \label{req:frame-indifference}
    $\W(\R\F) = \W(\F) \quad \forall \F \in \GLplus, \R \in \SO$.

    This expresses the requirement, also known as objectivity, that the material response is independent of rigid-body rotations of the current configuration. 
    Physically, the internal energy of the material does not change if the body is rotated as a whole. 

    \item \textbf{Material symmetry:} \label{req:symmetry}
    $\W(\F\R) = \W(\F) \quad \forall \F \in \GLplus, \R \in \SymmetryGroup$,

    where $\mathcal G \subseteq \SO$ is the material symmetry group, i.e., the set of all rotations that leave the material response invariant. 
    For isotropic materials, $\mathcal G = \SO$.

    \item \textbf{Polyconvexity:} \label{req:polyconvexity}
    $\W(\F) = \Wpoly(\F, \cof \F, \det \F)$ for some convex
    $\Wpoly \colon \mathbb{R}^{3 \times 3} \times  \mathbb{R}^{3 \times 3} \times \mathbb{R} \to \mathbb{R} \cup \{+\infty\}$.

    This condition, together with suitable growth and coercivity conditions (see next points), ensures existence of weak solutions to the equilibrium equation \eqref{eqn:equilibrium_strong_form}, as proved by John Ball \cite{ball1976convexity}. 
    Notice that polyconvexity can be relaxed to the weaker notion of quasiconvexity (a necessary condition for the existence of solutions) \cite{ball1976convexity}, but the latter is untractable and difficult to verify in practice, as it is a global property which cannot be reformulated as a local one \cite{kristensen1999non}.

    \item \textbf{Extreme compression growth:} \label{req:growth}
    $\W(\F) \to +\infty$ as $\J \to 0^+$.
    
    This condition ensures physically realistic behavior under extreme compression and prevents interpenetration of matter, reflecting that infinite energy would be required to collapse a volume element to zero.

    \item \textbf{Coercivity:} \label{req:coercivity}
    $\W(\F) \geq c_0(|\F|^2 + |\cof\F|^{3/2}) + c_1
    \quad \forall\,\F\in\GLplus \quad$ for some $c_0 > 0, c_1 \in \mathbb{R}$.

    Similarly to the previous condition, this ensures that the extreme strains should be maintained by infinite stresses.
    From the mathematical viewpoint, this condition is key for the application of the direct method of the calculus of variations to the proof of existence, as it guarantees that minimizing sequences remain bounded and that a weakly convergent subsequence exists, allowing the attainment of a minimizer.
    The coercivity condition considered in this work is a weaker version of the original one used by \cite{ball1976convexity}, and is due to M{\"u}ller \cite{muller1994new}.

\end{enumerate}

The frame indifference requirement \ref{req:frame-indifference} is equivalent to asking that $\W$ can be expressed as a function of the right Cauchy-Green tensor $\C = \F^T\F$.
Furthermore, for isotropic material ($\mathcal G = \SO$ in \ref{req:symmetry}), $\W$ can be expressed as a function of the invariants $\invI = \trace\C$, $\invII = \frac{1}{2}(\trace(\C)^2 - \trace(\C^2))$ and $\J$, that is there exists a function 
$\Wisoinv\colon \RplusZ \times \RplusZ \times \RplusZ \to \mathbb{R} \cup \{+\infty\}$ such that $\W(\F) = \Winv(\invI,\invII,\J)$.

\subsubsection{Structure preserving general form}

Several neural network architectures have been proposed in the literature to satisfy subsets of the requirements outlined above; however, to the best of our knowledge, no practically usable architecture has yet been shown to satisfy {the complete set \ref{req:residual-energy}--\ref{req:coercivity}, including the M{\"u}ller-type coercivity condition \ref{req:coercivity}, simultaneously by construction}.
For instance, the approach in \cite{fernandez2021anisotropic} guarantees frame indifference and material symmetry, but polyconvexity may be violated. 
In \cite{vlassis2020geometric}, some of the constraints are enforced only weakly through penalty terms in the loss function and are therefore not guaranteed throughout the training process. 
The work of \cite{klein2022polyconvex} represents the first neural-network-based constitutive model in which polyconvexity is enforced by design; however, this is achieved at the expense of exactly satisfying either frame indifference or the residual-stress-free condition.
In \cite{tac2022data}, Neural ODEs are employed to guarantee convexity with respect to individual invariants, but under the assumption of an additive decomposition of the energy into terms depending on a single invariant, which precludes the representation of energies involving interactions among multiple invariants. 
In \cite{thakolkaran2022nn}, polyconvexity and the residual-stress-free condition are jointly enforced; nevertheless, coercivity is not guaranteed by construction. 
Moreover, a convex and monotonically increasing dependence on the Jacobian~$\J$ is imposed, whereas convexity alone would suffice, thereby unnecessarily restricting the admissible hypothesis space.
{Furthermore, existence of solutions to the equilibrium problem generally requires a suitable coercivity condition, such as the M{\"u}ller-type condition \ref{req:coercivity} \cite{muller1994new}. 
Previous works, including \cite{linden2023neural}, which is, to the best of our knowledge, the closest existing architecture to the one proposed here, have considered weaker volumetric growth conditions, requiring the strain-energy density to diverge as $\J\to0^+$ and $\J\to+\infty$.
While such conditions provide an important control on extreme volumetric deformations, they do not in general yield coercive control with respect to $\F$ and $\cof\F$, as required in classical existence results for nonlinear elasticity \cite{ball1976convexity}.}

To define a general strain energy formulation that fully preserves the structure imposed by the requirements listed in the previous section, we observe that, when the strain energy density is expressed as a function of the invariants, the Piola--Kirchhoff stress tensor admits the following representation:

\begin{equation}
    \begin{split}
    \Piola(\F) &= 
    \frac{\partial \W(\F)}{\partial \F} =
    \frac{\partial \Winv}{\partial \invI } \frac{\partial \invI}{\partial \F} +
    \frac{\partial \Winv}{\partial \invII} \frac{\partial \invII}{\partial \F} +
    \frac{\partial \Winv}{\partial \J    } \frac{\partial \J}{\partial \F}
    \\
    &=
    \frac{\partial \Winv}{\partial \invI } 2 \F +
    \frac{\partial \Winv}{\partial \invII} 2 (\invI\F - \F\F^T\F) + 
    \frac{\partial \Winv}{\partial \J    } \cof \F
    \end{split}
\end{equation}
It follows that, in the reference configuration $\F=\Id$:
\begin{equation} \label{eqn:piola_reference}
    \Piola(\Id) = \left.\left(
    2 \frac{\partial \Winv}{\partial \invI } +
    4 \frac{\partial \Winv}{\partial \invII} + 
    \frac{\partial \Winv}{\partial \J    }
    \right)\right|_{\F = \Id} \Id
\end{equation}
Hence a practical way of guaranteeing the residual stress free condition \ref{req:residual-stress-free} is to let the term inside the parenthesis in \eqref{eqn:piola_reference} vanish.

Based on the above observations, we express the candidate strain energy density functionals in the general form of the next theorem.

\begin{theorem} \label{prop:W_full}
    Consider the strain energy density functional:
    \begin{equation} \label{eqn:W_full}
        \Winv(\invI,\invII,\J) = \Wbase(\invI,\invII,\J) + \Wgrow(\J) + \stressfactor(\J-1) - \W_0, 
    \end{equation}
    where:
    \begin{itemize}
        \item $\Wbase$ is a polyconvex energy density;

        \item     
        $\Wgrow\colon\RplusZ\to \RplusZ$ is a convex, $\Cont{1}$ function such that $\Wgrow(1) = \Wgrow'(1) = 0$.

        \item $\stressfactor \in \mathbb{R}$ is a constant defined as:
        $
            \stressfactor = -\left.\left(
            2 \frac{\partial \Wbase}{\partial \invI } +
            4 \frac{\partial \Wbase}{\partial \invII} + 
            \frac{\partial \Wbase}{\partial \J    }
            \right)\right|_{\F = \Id}
        $;

        \item $\W_0 \in \mathbb{R}$ is a constant defined as: $\W_0 = \left.\Wbase\right|_{\F = \Id}$.

    \end{itemize}
    Then, the material model \eqref{eqn:W_full} satisfies the requirements \ref{req:residual-energy}--\ref{req:polyconvexity}.
\end{theorem}

\begin{proof}
    Requirements \ref{req:residual-energy}--\ref{req:residual-stress-free}--\ref{req:frame-indifference}--\ref{req:symmetry} are easy to verify.
    Polyconvexity follows from the fact that the sum of polyconvex functions is polyconvex, and that a convex function of $\J$ is trivially polyconvex.
\end{proof}

\cref{prop:W_full} provides a systematic mechanism to enforce most of the requirements listed in Sec.~\ref{sec:W_requirements}, while retaining substantial flexibility in the choice of~$\Wbase$. In particular, the construction reduces the design of the constitutive model to ensuring polyconvexity of~$\Wbase$ and the satisfaction of the associated growth and coercivity conditions for the resulting strain energy density functional.

\begin{remark} \label{rem:why_stress_factor_on_J}
    To ensure the residual stress free requirement \ref{req:residual-stress-free}, the term $\stressfactor(\J-1)$ could be equivalently replaced either by $\frac{1}{2}\stressfactor(\invI-3)$ or by $\frac{1}{4}\stressfactor(\invII-3)$, or by linear combinations of these three terms.
    However, this would render more tricky to ensure polyconvexity.    
    Indeed, polyconvexity requires that the strain energy be a convex function of $\F$, $\cof \F$, and $\det \F$. 
    While $\J$ is directly one of the polyconvex arguments and thus adding a term proportional to $(\J-1)$ preserves polyconvexity regardless of the sign of the coefficient, the invariants $\invI$ and $\invII$ are convex functions of $\mathbf{F}$ and $\cof \mathbf{F}$, respectively. 
    Hence, polyconvexity is preserved provided that $\invI$ and $\invII$ are composed with a convex non-decreasing function (see \ifappendix\cref{prop:convex_nondecreasing_composition}\else Supplementary Information\fi). 
    However, ensuring that the corresponding coefficient $\omega$ is non-negative is not straightforward if $\Wbase$ is allowed to take general expressions.
\end{remark}

\subsubsection{Polyconvex Neural Networks}

We present here a neural network architecture that satisfies polyconvexity by design.
Similarly to other architectures proposed in the literature \cite{klein2022polyconvex,linden2023neural}, the one considered in this work is inspired by Input-Convex Neural Networks (ICNNs), introduced in \cite{amos2017input}, which are built upon two fundamental principles: (i) non-negative weighted sums of convex functions are themselves convex, and (ii) the composition of a convex function with a convex, non-decreasing function preserves convexity.

\begin{definition} \label{def:polyconvexNN}
    A Polyconvex Neural Network (PNN) with $\numLay \geq 1$ hidden neuron layers and $\numNeur{1},\dots,\numNeur{\numLay}$ neurons per layer is a function $\Wpnn\colon \RplusZ \times \RplusZ \times \RplusZ \to \mathbb{R}$, mapping $(\invI, \invII, \J)\mapsto \neurOut$ as recursively defined below:
    \begin{equation} \label{eqn:Wnn}
        \left\{
        \begin{aligned}
            \neur_{1} & = \act(   \weightI{1}(\invI - 3) 
                                + \weightII{1}(\invII - 3)
                                + \weightJ{1}(\J - 1) 
                                + \bias{1}
                                ), &\\
            \neur_{i} & = \act(   \weightI{i}(\invI - 3) 
                                + \weightII{i}(\invII - 3) 
                                + \weightJ{i}(\J - 1) 
                                + \weightZ{i-1} \neur_{i-1}
                                + \bias{i} 
                                ), &\quad\text{for } i = 2, \ldots, \numLay\\
            \neurOut & = \Wconst \, \weightOUT \cdot \neur_{\numLay}.&
        \end{aligned}
        \right.
    \end{equation}
    Here, $\neur_{i} \in \mathbb{R}^{\numNeur{i}}$ denotes the vector of neuron activations of the $i$-th layer, for $i = 1, \ldots, \numLay$.    
    The activation function $\act\colon \mathbb{R} \to \mathbb{R}$ is a convex, non-decreasing, non-constant and $\Cont{1}$ function, while $\Wconst > 0$ is a fixed scaling constant.
    The trainable parameters of the network are
    $\weightI{i}, \weightII{i}, \weightJ{i}, \bias{i} \in \mathbb{R}^{\numNeur{i}}$, for $i = 1, \ldots, \numLay$;
    $\weightZ{i} \in \mathbb{R}^{\numNeur{i+1} \times \numNeur{i}}$, for $i = 1, \ldots, \numLay-1$;
    $\weightOUT \in \mathbb{R}^{\numNeur{\numLay}}$.
    By definition, the parameters must satisfy the positivity constraints:
    $\weightI{i} \succ 0$, 
    $\weightII{i} \succ 0$, 
    $\weightZ{i} \succ 0$, 
    $\weightOUT \succ 0$ for any $i$.
\end{definition}

Following the design of ICNNs \cite{amos2017input}, two architectural choices are made:
first, non-negativity is required for most of the weights;
second, the activation function $\act$ is chosen to be a convex and non-decreasing.
As we prove below (see \cref{prop:polyconvexity}), these two choices combined together ensure that $\Wpnn$ is polyconvex.
We notice that the weight that directly multiplies $\J$ is not required to be nonnegative.
In fact, as observed in \cref{rem:why_stress_factor_on_J}, while $\invI$ and $\invII$ are convex functions of the triplet $(\F, \cof\F, \J)$, and hence can be multiplied only by nonnegative weights to preserve polyconvexity, $\J$ is part of the triplet itself, and so can safely be multiplied by negative weights.

We remark that, for the purpose of ensuring polyconvexity, non-negativity of the above mentioned weights would be sufficient.
However, to guarantee the coercivity requirement too, we require strict positivity, as it will be clarified later.

The constant $\Wconst > 0$ is a dimensional hyperparameter necessary to recover the physical dimension of an energy density (e.g. $\si{\pascal}$). 
Its value could be set based on a prior estimate of the material overall stiffness, to provide the model with the suitable order of magnitude.

We notice that, for the sake of generality, we have included skip connections in \eqref{eqn:Wnn}, allowing each layer to directly depend on the invariants $\invI$, $\invII$ and $\J$, through the weights $\weightI{i}$, $\weightII{i}$ and $\weightJ{i}$ for $i \geq 2$.
However, these terms can be removed without compromising the validity of the theoretical results stated below.

\begin{theorem} \label{prop:polyconvexity}
    Let us consider the strain energy density functional $\Wpnn$ of \cref{def:polyconvexNN}.
    Then, for any choice of the neural network hyperparameters and parameters satisfying the positivity constraints listed in the definition, $\Wpnn$ is polyconvex.
\end{theorem}
\begin{proof}
    First, we observe that
    \begin{equation} \label{eqn:invariants_and_frobenius}
        \begin{split}
        \invI &  = \trace\C = \trace(\F^T\F) = |\F|^2, \\ 
        \invII & = \trace(\cof\C) = \trace((\cof\F)^T\cof\F) = |\cof\F|^2 ,
        \end{split}
    \end{equation}
    where we have exploited the identity $\cof\C = J^2 \C^{-T} = J^2 \F^{-1} \F^{-T} = (\cof\F)^T\cof\F$.
    This shows that $\invI$ and $\invII$ are convex functions of $\F$ and $\cof\F$, respectively. 
    
    Each entry of the argument of $\act$ in the first layer is polyconvex, being the sum of polyconvex functions.
    Hence, each neuron of the first layer in the network is also polyconvex, being the composition between a polyconvex function with a convex non-decreasing function.
    By similar arguments, also neurons in successive layers are polyconvex. 
    Finally, $\Wpnn$ is polyconvex being the positive-weighted sum of polyconvex functions.
\end{proof}

\subsubsection{Hyperelastic Neural Networks} \label{sec:HNN}

We are ready to introduce the hypothesis space $\spaceW$ used within this work.

\begin{definition}
    A Hyperelastic Neural Network (HNN) is a function $\Whnn \colon \RplusZ \times \RplusZ \times \RplusZ \to \mathbb{R}$, defined as in \eqref{eqn:W_full}, by setting $\Wbase = \Wpnn$, with $\Wpnn$ being a PNN, and where $\stressfactor$ and $\W_0$ are defined as in \cref{prop:W_full}.
\end{definition}

The following results show that, for an HNN, all the requirements \ref{req:residual-energy}--\ref{req:coercivity} are satisfied by design, provided that the volumetric term $\Wgrow$ satisfies two additional hypothesis (growth at $\J \to 0^+$, and superlinear growth at $\J \to +\infty$).

\begin{theorem} \label{prop:W_full_nn}
    Let us consider a HNN, where $\Wgrow$ satisfies the hypotheses of \cref{prop:W_full} with the additional hypotheses:
    \begin{equation} \label{eqn:Wgrowth_limits}
            \lim_{\J \to 0^+} \Wgrow(\J) = +\infty, \qquad
            \lim_{\J \to +\infty} \frac{\Wgrow(\J)}{\J} = +\infty.
    \end{equation}
    Then, for any admissible choice of the neural network hyperparameters and parameters, the strain energy density functional satisfies all the requirements \ref{req:residual-energy}--\ref{req:coercivity}.
\end{theorem}
\begin{proof}    
    Requirements \ref{req:residual-energy}--\ref{req:polyconvexity} follow from \cref{prop:W_full} and \cref{prop:polyconvexity}.
    To prove requirement \ref{req:growth}, it suffices to observe that $\Wpnn$ is non-decreasing in both $\invI$ and $\invII$, being the composition of non-decreasing functions. 
    Therefore $\Wpnn(\invI,\invII,\J) \geq \Wpnn(0,0,\J)$, which has clearly a finite limit when $\J \to 0^+$.
    Hence, thanks to \eqref{eqn:Wgrowth_limits}, requirement \ref{req:growth} holds true.

    To prove coercivity (requirement \ref{req:coercivity}), we start by observing that, as $\act$ is $\Cont{1}$ convex and non-decreasing but non-constant, we have that $\act(y) \geq ay+b$ for some $a>0$ and $b\in\mathbb{R}$.
    It follows that the first layer of $\Wpnn$ satisfies:
    \begin{equation}
        \exists \, a > 0, b \in \mathbb{R}, c \in \mathbb{R} \text{ s.t. } \quad 
        \neur_{1} \succeq a (\invI + \invII) + b \J + c.
    \end{equation}
    By the same argument, all successive layers $\neur_{i}$ satisfy the same inequality.
    Therefore, we have:    
    \begin{equation}
        \exists \, a > 0, b \in \mathbb{R}, c \in \mathbb{R} \text{ s.t. } \quad 
        \Wpnn(\invI,\invII,\J) \geq a (\invI + \invII) + b \J + c.
    \end{equation}
    Then:
    \begin{equation}
        \exists \, a > 0, b \in \mathbb{R}, c \in \mathbb{R} \text{ s.t. } \quad 
        \Whnn(\invI,\invII,\J) \geq a (\invI + \invII) + b \J + \Wgrow(\J) + c.
    \end{equation}
    The function $\J \mapsto b \J + \Wgrow(\J)$ is continuous and tends to $+\infty$ both for $\J \to 0^+$ and $\J\to+\infty$ (thanks to the super-linear growth of $\Wgrow$), and hence it is bounded from below. It follows that
    \begin{equation}
        \exists \, a > 0, b \in \mathbb{R}, c \in \mathbb{R} \text{ s.t. } \quad 
        \Whnn(\invI,\invII,\J) \geq 
        a (\invI + \invII) + c 
        = a (|\F|^2 + |\cof\F|^2) + c,
    \end{equation}
    where we have exploited \eqref{eqn:invariants_and_frobenius}. 
    This entails the coercivity of the strain energy density functional.

\end{proof}

The term $\Wgrow(\J)$ determines the behavior of the energy for $\J \to 0^+$ and $\J \to +\infty$.
In principle, we could think of learning $\Wgrow(\J)$, by designing an architecture that satisfies all the associated hypotheses.
However, this would be poorly constrained by the training data, as both extreme regimes of volumetric compression and expansion are not observed in experiments.
Indeed, the role of $\Wgrow(\J)$ is mainly to ensure {existence of solutions to the equilibrium problem} and promote the finite element solvability along the whole optimization process.
Therefore, we prefer to choose a priori the shape of $\Wgrow(\J)$, and limit ourselves to learn a multiplicative parameter.
Specifically, we set $\Wgrow(\J) = \frac{1}{2} \weightGrowth (\J-1)\log\J$, where $\weightGrowth > 0$ is a trainable parameter.
Notice that this choice is very conservative, as it has a very mild growth both for $\J \to 0^+$ (logarithmic) and for $\J \to +\infty$ ($\sim\J\log\J$, considering that superlinear growth is required by \cref{prop:W_full_nn}).

\begin{remark} \label{rem:isochoric_invariant_input}
    An alternative to the architecture described in \cref{def:polyconvexNN} consists in replacing the invariants $\invI$ and $\invII$ with their isochoric counterparts. More precisely, we can replace $\invI$ by $\isoinvI$. 
    However, it is not convenient to use $\isoinvII$, since it is not polyconvex (see Lemma~2.4 in \cite{hartmann2003polyconvexity}), and therefore the proof of \cref{prop:polyconvexity} would no longer apply. As an alternative, one may use $\isoinvII^{3/2}$ as an input of the neural network, since this quantity is proven to be polyconvex (Corollary~2.3 in \cite{hartmann2003polyconvexity}). 
    With this choice, the result of \cref{prop:polyconvexity} remains valid.
    On the other hand, the corresponding HNN is no longer strictly coercive in the sense of requirement~\ref{req:coercivity}, but only coercive on sublevel sets of the form $(\J < J_{\text{max}})$, as can be readily shown by revisiting the proof of \cref{prop:W_full_nn}. 
    Nevertheless, the volumetric term $\Wgrow(\J)$ prevents excessive growth of $\J$ in practice, and therefore this formulation can still be safely employed in applications.
\end{remark}

\subsection{Initialization and parametrization of HNNs}

The initialization of trainable parameters plays a crucial role in the successful training of any neural network architecture.
Poor initialization may result in slow convergence or cause the optimizer to become trapped in undesirable regions of the parameter space~\cite{glorot2010understanding,lecun2002efficient}.
Whenever a neural architecture is introduced, a carefully designed initialization strategy is therefore essential.
Typically, initialization schemes aim to control the magnitude of activations and gradients, avoiding both explosion and collapse, while preventing saturation of the activation functions, typically by keeping pre-activations in a neighborhood of zero~\cite{glorot2010understanding,he2015delving}.

Moreover, in the HNN architecture several weights are constrained to assume positive values in order to preserve polyconvexity.
Since most training algorithms are designed for working in unconstrained parameter spaces, a reparametrization is required to convert the constrained optimization problem into an unconstrained one.
In this section, we describe the proposed strategy that jointly addresses the parametrization and initialization of the HNN trainable parameters.

First, we select an activation function satisfying $\act(0) = 0$. 
Specifically, throughout this work, we employ the shifted softplus function $\act(z) = \log( 1 + e^z) - \log(2)$, which is convex, non-decreasing, $\Cont{1}$, non-constant, and therefore satisfies the assumptions of \cref{def:polyconvexNN}.
Moreover, we initialize all biases $\bias{i}$ to zero.
These conditions together ensure that, in the rest configuration (i.e. $\F = \Id$), all the activations in Eq.~\eqref{eqn:Wnn} vanish.
Consequently, around the rest configuration, the activation function operates in the vicinity of $z\approx 0$, where its response is well-conditioned.

Controlling the activations alone, however, does not guarantee that the initialized constitutive model exhibits physically meaningful properties.
Indeed, since the Piola-Kirchhoff stress tensor is obtained by differentiating the strain energy density $\W$ with respect to the deformation gradient $\F$, it is essential to ensure that the gradients of the network with respect to its inputs neither vanish -- leading to an excessively compliant material -- nor blow up -- resulting in an unrealistically stiff response.
This issue is particularly delicate for the derivatives with respect to the invariants $\invI$ and $\invII$, whose associated coefficients are constrained to be non-negative:
because the linear layers also involve non-negative weights, successive layers may accumulate positive contributions, potentially amplifying gradients as the number of neurons increases.
To counteract this effect, we scale each linear layer by the inverse of the number of incoming neurons, as we detail below.

Moreover, to enforce positivity of the constrained weights, we adopt a standard reparametrization in which each constrained weight is expressed as the composition of an unconstrained trainable variable with a smooth positivity-enforcing function $\pos \colon \mathbb{R} \to \Rplus$.
Specifically, we use the softplus function $\pos(w) =  \log( 1 + e^w)$.

In conclusion, the resulting parametrization (for $i=1,\dots,\numLay$) reads:
\begin{equation}
    \weightI{i} = \pos(\BaseWeightI{i})
    , \qquad
    \weightII{i} = \pos(\BaseWeightII{i})
    , \qquad
    \weightZ{i} = \frac{1}{\numNeur{i}} \pos(\BaseWeightZ{i})
    , \qquad
    \weightOUT = \frac{1}{\numNeur{\numLay}} \pos(\BaseWeightOUT)
    ,
\end{equation}
where $\BaseWeightI{i}$, $\BaseWeightII{i}$, $\weightJ{i}$, $\BaseWeightZ{i}$, $\BaseWeightOUT$ are initialized independently according to Gaussian distributions with zero mean and standard deviation $\stdInit > 0$. 
Similarly, we parametrize $\weightGrowth$ as $\pos(\BaseWeightGrowth)$, with $\BaseWeightGrowth$ initialized to zero.

In \ifappendix Appendix~\ref{sec:initialization_proof}\else the Supplementary Information\fi, we show analytically that, thanks to the proposed scaling and parametrization, the expected value of the derivatives of the hyperelastic energy with respect to the invariants, evaluated at the reference configuration, is kept under control (i.e., it does not vanish nor it blows up) as the number of neurons increases.

\subsection{Neural-DFEM}

The material discovery method considered in this work is based on the formulation \eqref{eqn:discovery_DFEM} wherein the HNN architecture described in Sec.~\ref{sec:HNN} is employed to define the hypothesis space $\spaceW$.
The HNN parameters are learned by minimizing the loss function \eqref{eqn:discovery_DFEM}, whose evaluation involves the solution of a set of finite element equilibrium problems, one for each training experiment.
Crucially, the training procedure exploits a differentiable finite element solver, allowing gradients to be propagated through the equilibrium solution itself. 
As a result, the parameter updates explicitly account for the effect of the constitutive model on the mechanical response of the training specimens.
We refer to this approach as Neural Differentiable Finite Element Method (Neural-DFEM).

The practical realization of this framework entails a number of nontrivial challenges that must be addressed to ensure efficiency and robustness. In the following, we outline these challenges and describe how they are resolved within the Neural-DFEM framework.

\begin{itemize}
    \item The first challenge concerns the differentiation of the finite element solver.
    While automatic differentiation is employed to compute the derivatives of the HNN constitutive model with respect to trainable parameters, the equilibrium equations introduce an implicit dependence of the displacement field on these parameters.
    Since the governing equilibrium problem is nonlinear and solved through an iterative Newton method, differentiating the solver iterations directly would be computationally inefficient.
    Instead, we compute the required sensitivities by solving the algebraic adjoint system associated with the discretized equilibrium equations \cite{farrell2013automated}.
    This hybrid strategy combines automatic differentiation for the neural constitutive model with adjoint-based differentiation through the finite element solver, requiring the solution of one sparse linear system per training experiment to obtain the gradients with respect to the trainable constitutive parameters.
    {This strategy, together with some considerations on the computational cost, is detailed in Appendix~\ref{sec:sensitivities_computation}.}

    \item A second challenge arises from the computational cost associated with the evaluation of the loss function, which requires solving multiple finite element problems -- one for each training experiment -- at every training epoch. 
    This may potentially lead to a significant per-epoch cost.
    To mitigate this effect, we reduce the total number of training epochs by adopting a quasi-Newton optimization strategy, specifically the BFGS method \cite{nocedal2006numerical}.
    This choice is particularly effective in the present setting, as the neural constitutive model operates in a low-dimensional space (e.g., the space of strain invariants), resulting in a relatively small number of trainable parameters. 
    Consequently, the linear systems involved in the BFGS updates do not constitute a computational bottleneck.
    As a stopping criterion, the optimization is terminated when the relative improvement of the loss over the last 50 iterations falls below $10^{-4}$.

    In addition, the cost of the initial training phase is controlled through a load continuation strategy, which is used to compute a suitable initial equilibrium solution. 
    During subsequent training epochs, the finite element solver is warm-started using the equilibrium solution from the previous epoch. 
    Since the constitutive response evolves smoothly over the course of training, the equilibrium solution changes only marginally between consecutive epochs, leading to a modest cost for each finite element problem.

    \item A further challenge stems from the fact that the loss function depends on the solution of the equilibrium problem, which must remain {solvable} throughout the entire optimization process, from the initial parameter guess to the final learned model. 
    This requires {existence of solutions} to be ensured over the whole hypothesis space explored during training.
    In Neural-DFEM, this requirement is satisfied by the HNN architecture introduced in Sec.~\ref{sec:HNN}. 
    As proved therein, the HNN guarantees physical meaningfulness and {existence of equilibria} for all admissible parameter values.

    \item Despite the {existence of solutions to the equilibrium problem}, convergence of the finite element solver cannot be guaranteed in all cases, as the Newton method features only local convergence. 
    Two strategies can, in principle, be adopted to address this issue.
    A first option consists in introducing a continuation procedure to gradually transition between successive equilibrium solutions. 
    While effective, this approach increases the computational complexity of the training process.
    In Neural-DFEM, we instead adopt a more conservative strategy, to prevent parameter updates from moving the system outside the region in the parameter space of approximate linearity, where gradient-based directions remain reliable. 
    This is achieved by leveraging the backtracking mechanism embedded in the line search of the BFGS algorithm, so that trial steps leading to non-convergence of the Newton solver are automatically rejected. 
    As a result, the optimization proceeds while remaining within the basin of attraction of the equilibrium solution, without requiring additional continuation steps during training.

\end{itemize}

\section{Results}
\label{sec:results}

To investigate the capabilities of the proposed methods, we consider several synthetic test cases, starting from two-dimensional domains deformed under plane strain conditions (see Fig.~\ref{fig:test_cases_2D}), then moving to three-dimensional domains (see Fig.~\ref{fig:test_cases_3D}). 
The full description of the test cases is reported in Appendix~\ref{sec:test_cases}.

To generate training and testing data, we employ the finite element method, and we compute numerically the equilibrium configuration resulting from known hyperelastic models. Specifically, we consider three ground truth material models, namely the Ishihara (IH) \cite{ishihara1948statistical}, the Mooney-Rivlin (MR) \cite{mooney1940theory,rivlin1948large}, and the Fung (FU) \cite{fung2013biomechanics} hyperelastic constitutive laws (see Appendix~\ref{sec:material_models}). 
{In all numerical experiments, displacement observations comprise all components of the displacement vector and are sampled at finite-element mesh nodes. In the full-field setting, observations are available at all mesh nodes, whereas in the boundary-only setting they are available at all mesh nodes lying on the entire boundary $\partial\Omega$. Hence, the boundary-only observation operators considered here correspond to dense observations over the whole boundary. The mesh sizes employed for all test cases are reported in Appendix~\ref{sec:test_cases}.}

To take into account the measurement error that unavoidably affects any source of experimental data, we add an artificial independent Gaussian noise with prescribed standard deviation $\stdNoise$, that we will refer to as the noise level.
To quantify the discrepancy between model predictions and ground-truth data, we consider the root mean square error normalized by the displacement variance square root (vRMSE): by construction, a naïve predictor that outputs the mean value at each point attains vRMSE~=~1, whereas a perfect predictor with zero error yields vRMSE~=~0 (see \ifappendix Appendix~\ref{sec:metrics} \else Supplementary Information \fi for precise definitions).


\subsection{Plane strain deformations with full displacement field observation}

We first consider the case of two-dimensional domains with deformations taking place under plane strain conditions. Here we assume that the displacement field is measured over the whole domain (full displacement field observation). 
We take as benchmark cases the mechanical setups proposed in \cite{flaschel2021unsupervised}, and later considered in subsequent works \cite{flaschel2022discovering,flaschel2023automated,thakolkaran2022nn}.
Specifically, the models are trained based on the displacement field recorded from a square plate with a circular hole, subjected to an asymmetrical displacement-controlled biaxial tension with varying loading, hereafter referred to as Setup~1 (Fig.~\ref{fig:test_cases_2D}a).
The discovered models are then tested on a different geometry, namely a square plate with two ellipsoidal holes, and subjected to displacement-controlled uniaxial tension in the vertical direction, while horizontal sides are stress-free (Setup~2, see Fig.~\ref{fig:test_cases_2D}b).

\ifappendix\figureSetupTwoD\fi
\ifappendix\else\figureSetupTwoD\fi

First, we consider the IH model with a noise magnitude $\stdNoise = 10^{-3}$, representative of modern DIC setups \cite{thakolkaran2022nn}.
For the proposed method, we tune the hyperparameters using a simple grid search, selecting the configuration with the lowest training loss (more details in Appendix~\ref{sec:hyperparameters_tuning}). 
Notably, the presence of strong structural priors and the consequent robustness to overfitting allows us to rely on this straightforward selection criterion, eliminating the need for cross-validation, which is typically required to control overfitting.
For the selected architecture, we train 5 models by varying the random seed used to initialize the trainable parameters, and we select the model with lowest training loss.

As a benchmark in the full-field setting, we compare against EUCLID \cite{thakolkaran2022nn}, a representative VFM-based stress-label-free model-discovery method.
To provide a further reference, we also consider a traditional identification approach in which the functional form of the hyperelastic law is fixed \emph{a priori} -- here chosen as Neo-Hookean -- and only its material parameters are calibrated on the training data.
To provide a fair comparison, we consider the same training data and the same optimization algorithm as for Neural-DFEM.
In Fig.~\ref{fig:results_base_2D}a, we report boxplots of the error, obtained on the test specimen, as a function of the applied load.
The results show that the model discovery methods (i.e. VFM and Neural-DFEM) achieve much better results than the calibrated Neo-Hookean model.
Notably, Neural-DFEM consistently provides the most accurate results across all load values, with an average vRMSE of $\num{2.16e-3}$ against $\num{1.36e-2}$ of the VFM method (6.3 times smaller).

\begin{figure}
    \includegraphics[trim=0 14.58cm 0 0, clip, page=3]{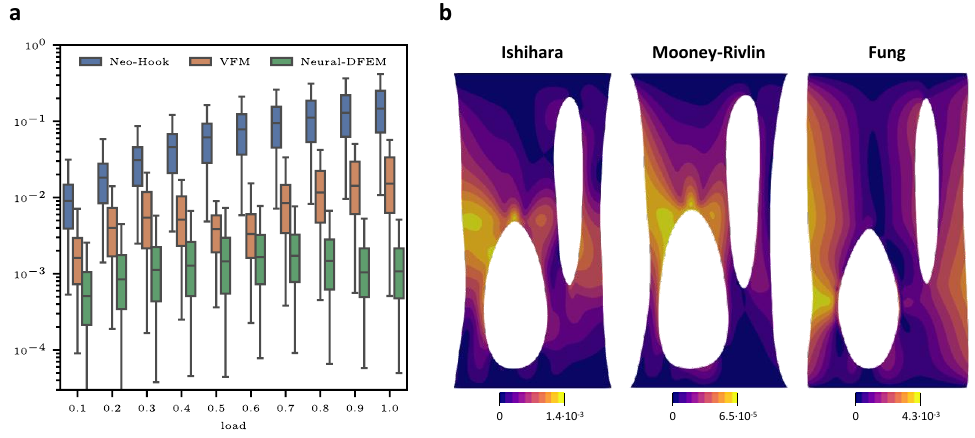}
    \vspace{-.9cm}
    \caption{
        \textbf{Results of 2D test cases (training on Setup~1, testing on Setup~2).} 
        (a) Boxplots of the displacement errors, normalized by the standard deviation of the ground-truth displacement, shown as a function of the applied load.
        The training data are generated using Setup~1 with $\stdNoise = 10^{-3}$, while the test data correspond to Setup~2.
        Three approaches are compared: classical parameter fitting of a Neo-Hookean material; EUCLID~\cite{thakolkaran2022nn}, based on the VFM; and the proposed Neural-DFEM.
        (b) Spatial distribution of the local displacement error for Setup~2 obtained with Neural-DFEM at the largest applied load (i.e., $\delta = 1$) and with $\stdNoise = 10^{-3}$, for the three material models considered. Note the different color scales used in the three cases.
        }
    \label{fig:results_base_2D}
\end{figure}

\subsection{Noise robustness}

Next, we investigate the noise robustness of Neural-DFEM, by considering noise levels up to two orders of magnitude larger than the one considered in the previous section (i.e. up to $\stdNoise = 10^{-1}$), for the three ground truth material models, namely IH, MR and FU.

In Fig.~\ref{fig:results_summary_2D}a we show the load-reaction curves of Setup~2, obtained for different noise levels and compared with the ground truth.
Except for the highest noise level in the FU case, the curves are practically indistinguishable, from a visual standpoint, from that obtained with the ground-truth model.

Besides Setup~2, we challenge the discovered models to predict the displacement under conditions that are even farther to the training ones. 
To this aim, we introduce Setup~3 (see Fig.~\ref{fig:test_cases_2D}c), consisting of 10 different perforated plates with a random number of holes -- ranging from 1 to 3 -- with random sizes and positions, as in \cite{zhang2025shape}.
On the sides of the plates, normal spring boundary conditions are applied, and two opposite sides are loaded with normal traction with varying magnitude, taking both positive (traction) and negative (compression) sign.

In Fig.~\ref{fig:results_summary_2D}b, we report boxplots of the errors in both the displacement and the reaction force, as a function of the noise level, for the three materials considered.
The figure shows the errors obtained not only for the training setup (Setup~1), but also for the two testing mechanical setups (Setups~2 and~3).
The results indicate that the constitutive laws discovered from Setup~1 generalize well to Setup~2 and Setup~3.
Indeed, for sufficiently low noise levels, a normalized error on the order of $10^{-2}$ or smaller is achieved in all cases.
As expected, the lower the noise in the training data, the more accurate is the model learned by Neural-DFEM.
Furthermore, Neural-DFEM exhibits good robustness to noise even at relatively high noise levels, which here reach up to two orders of magnitude higher than those considered reasonable in previous works and in typical applications~\cite{thakolkaran2022nn}.

\begin{figure}
    \includegraphics[trim=0 7.32cm 0 0, clip, page=4]{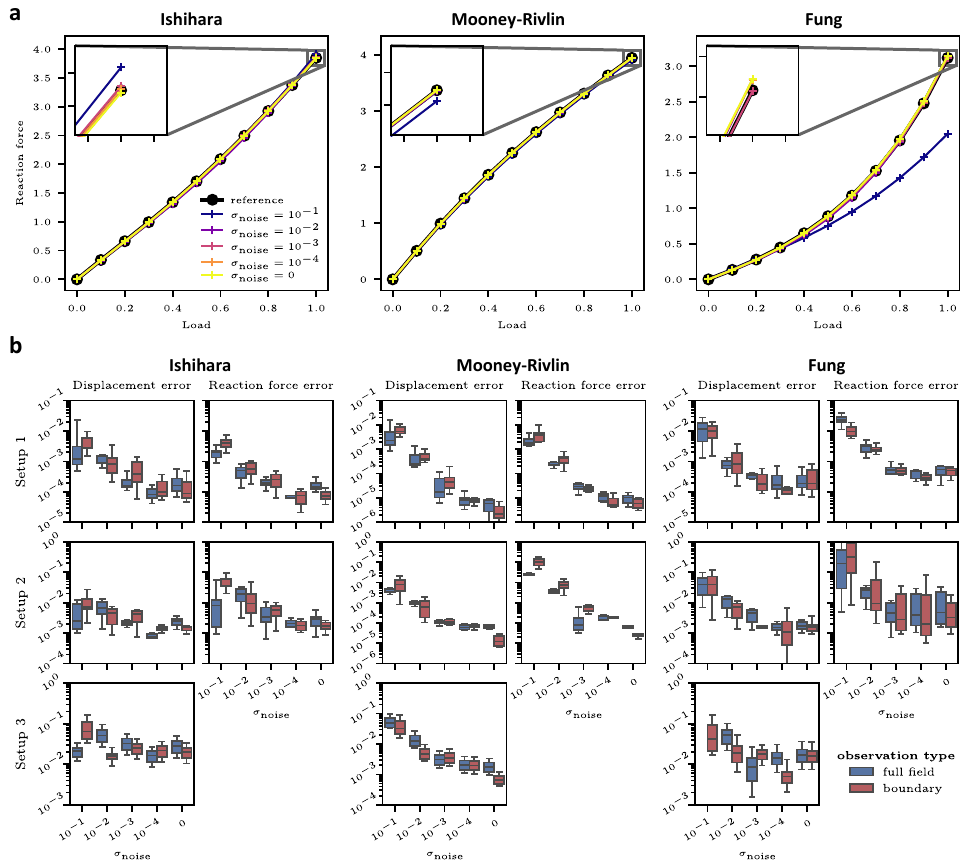}
    \caption{
        \textbf{Results of 2D test cases (training on Setup~1, testing on Setup~2 and Setup~3).} 
        (a) Predicted reaction force on the top boundary of Setup~2 as a function of the applied load and of the training data noise $\stdNoise$, compared with the reference reaction force obtained using the ground-truth material model.
        The insets show a magnified view of the force peak.
        (b) Boxplots of displacement and reaction force errors, both normalized by the standard deviation of the corresponding ground-truth, shown as a function of the training noise level $\stdNoise \in \{ 10^{-1},  10^{-2},  10^{-3},  10^{-4}, 0\}$. 
        Results are reported for the three material models considered (indicated at the top of each column) and for both the full-field and boundary observation cases (see the legend in the bottom-right panel).
        The three rows correspond to Setups~1--2--3, respectively. 
        Hence, while the first row shows the performance in the training setting (fitting regime), the second and third rows assess the ability of Neural-DFEM to generalize to mechanical setups different from those observed at training time.
        For Setup~3, no reaction force is defined, since no Dirichlet boundary conditions are prescribed.
        }
    \label{fig:results_summary_2D}
\end{figure}

Among the three materials, the highest accuracy is achieved for the MR model, suggesting that certain material behaviors may be more easily represented within the considered hypothesis space.
Moreover, generalization from Setup~1 to Setup~3 appears more challenging than to Setup~2, as indicated by the larger errors.
This aspect is further investigated in the next section.

\subsection{Distributional shift in strain-state space and generalization}

To shed light on the generalization capabilities of models discovered from a setup to a different setup, it is instructive to consider the strain state of the deformation involved.
Indeed, because the domain shape differs from one setup to another, to provide a portrait of the strain state that is comparable from one setup to another, we propose to consider the statistical distribution of the principal stretches ($\lambda_1$, $\lambda_2$).
As shown in Fig.~\ref{fig:test_cases_2D}d, the training data coming from Setup~1 mainly lay in a region comprised within the state of perfect biaxial tension ($\lambda_1 = \lambda_2 > 1$) and uniaxial tension ($\lambda_1 > 1$, $\lambda_2 = 1$).
This is not surprising, considering that the macroscopic tension state is an asymmetrical biaxial tension of the specimen.
In Setup~2, the macroscopic uniaxial tension translates at the microscopic level in a strain distribution that lies in proximity of the perfect uniaxial tension state.
As for Setup~3, the presence of compression load, not present in the first two setups, introduces a significant presence of microscopic strain states close to the perfect uniaxial compression state ($\lambda_1 = 1$, $\lambda_2 < 1$). 
Moreover, while Setup~1 and 2 employ a displacement-controlled load, Setup~3 is based on a tension-controlled load, with the loaded sides free to compress in the tangent direction due to the volumetric changes penalization of the materials, thus favouring the introduction of microscopic strain states close to the simple shear state ($\lambda_1 \cdot \lambda_2 = 1$).
As a consequence, the overall strain distribution in the specimens of the training setup (Setup~1) is closer to Setup~2 than to Setup~3. 
This explains the higher accuracy achieved in predicting the displacement for Setup~2 compared to Setup~3.

To quantify the distributional shift between training and test deformation regimes, we compute the Sinkhorn divergence between the empirical distributions of the principal stretches. 
The Sinkhorn divergence provides a measure of distributional mismatch based on entropy-regularized optimal transport \cite{cuturi2013sinkhorn}.
The results confirm that larger divergence values consistently correspond to test cases exhibiting stronger generalization gaps, thus supporting the interpretation that generalization performance is influenced by the coverage of the strain-state manifold during training.
In particular, comparing Setup~2 with Setup~3, the divergence increases from 0.17 to 0.32 for IH, from 0.24 to 0.39 for MR, and from 0.15 to 0.27 for the FU model. 

\subsection{Plane strain deformations with boundary observation}

We consider now the case when the training specimen displacement field is measured only at the boundary of the domain.
This case cannot be handled with VFM-based approaches, as the full displacement field should be available to evaluate the loss function.
Instead, the considered framework allows -- with a minor modification of the loss function with respect to the case with full-field observation -- to handle this case too.
We run again the same tests above, for the same noise levels.
The results, reported in Fig.~\ref{fig:results_summary_2D}b, show that a virtually identical accuracy as for the full-field observation condition is obtained in all the cases considered, both for Setup~2 and Setup~3.
More precisely, the vRMSE ratio between full-field and boundary observation cases is $1.03 \pm 0.67$ (mean $\pm$ standard deviation).
In the considered synthetic benchmarks, this result indicates that dense boundary displacement observations, combined with the measured global reactions, can contain sufficient information to identify a constitutive response with accuracy comparable to that obtained from full-field observations.

\subsection{3D specimens with boundary observations}

Finally, we consider three-dimensional specimens, which allow for a full range of deformation modes that were partially constrained in the previous plane-strain configurations.

For training, we consider Setup~4 (see Fig.~\ref{fig:test_cases_3D}a), a three-dimensional generalization of Setup~1, consisting of a cube containing a spherical inclusion and subjected to asymmetric triaxial displacement-controlled loads.
The learned models are then tested on two additional setups designed to challenge them under markedly different deformation regimes.
In Setup~5 (see Fig.~\ref{fig:test_cases_3D}b), the specimen is a cube with an oblique through-hole, subjected to a displacement-controlled load combining tension and torsion, while the remaining faces are stress-free.
In Setup~6 (see Fig.~\ref{fig:test_cases_3D}c), we consider a more complex, real-world-inspired geometry based on a three-dimensional mounting clip from an open-source CAD model~\cite{thingiverse_clip}. The clip is constrained by a roller-type condition on the bottom surface and subjected to combined follower and distributed forward loads. Additionally, normal-spring boundary conditions are imposed on selected surfaces to generate nontrivial internal stress states during loading. 
This setup is designed to assess generalization to markedly different geometries and mixed boundary conditions.

\ifappendix\figureSetupThreeD\fi
\ifappendix\else\figureSetupThreeD\fi

In Fig.~\ref{fig:results_summary_3D}a, we report the load-reaction curves obtained for Setup~5 at different noise levels with a boundary-only observation, while Figure~\ref{fig:results_summary_3D}b shows the boxplots of the displacement reconstruction error as a function of the noise level, both for the training setup and for the two testing setups.
The results show that Neural-DFEM successfully learns constitutive laws that generalize well across different configurations, also in the case of three-dimensional specimens. 
Remarkably, and consistently with the two-dimensional plane-strain experiments, the accuracy obtained with boundary-only observations is nearly identical to that achieved with full-field data (vRMSE ratio $1.03 \pm 0.41$).
Within these controlled benchmarks, the results indicate that the prescribed boundary-only observation operator, together with the available reaction data, can be sufficient for effective material discovery, even in fully three-dimensional settings.
Taking $\stdNoise = 10^{-3}$ and the boundary-only observation case as a reference, the vRMSE in Setups~5 and~6 is equal to $\num{6.79e-03}$ and $\num{1.54e-02}$ for IH, $\num{9.25e-04}$ and $\num{4.21e-04}$ for MR, and $\num{1.22e-03}$ and $\num{1.41e-03}$ for FU, respectively.
Across the considered noise levels, the method maintains low prediction errors, suggesting robustness to measurement noise.

\begin{figure}
    \includegraphics[trim=0 7.32cm 0 0, clip, page=5]{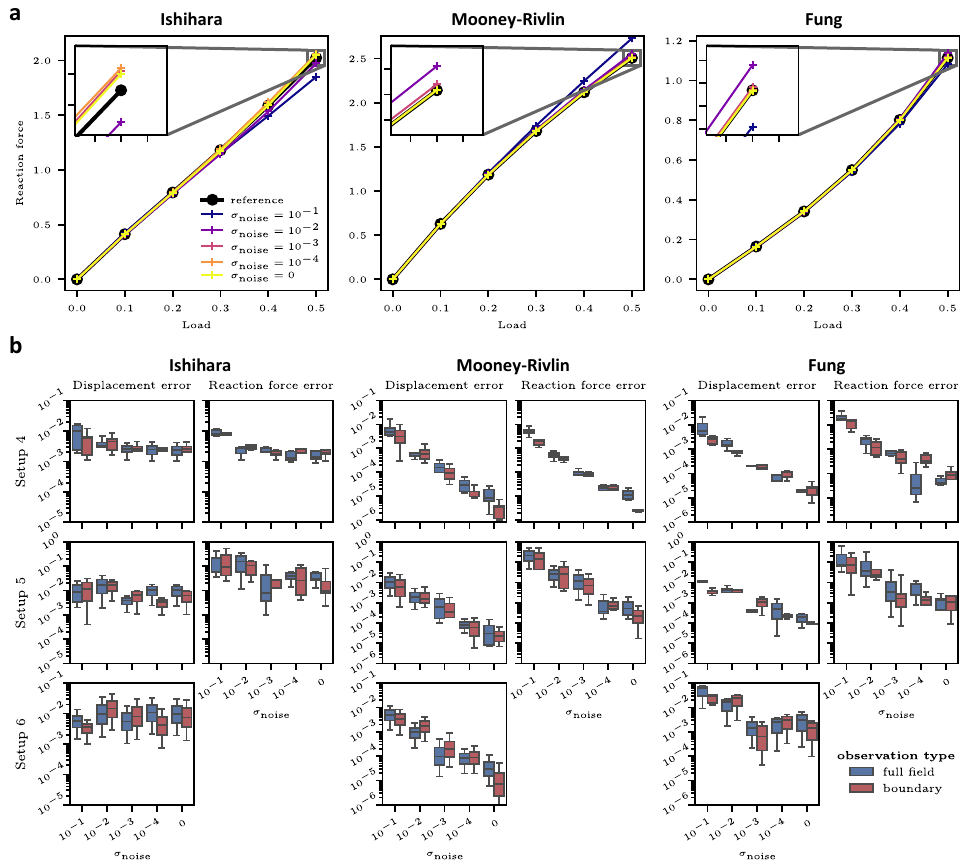}
    \caption{
        \textbf{Results of 3D test cases (training on Setup~4, testing on Setup~5 and Setup~6).} 
        (a) Predicted reaction force on the top boundary of Setup~5 as a function of the applied load and of the training data noise $\stdNoise$, compared with the reference reaction force obtained using the ground-truth material model.
        The insets show a magnified view of the force peak.
        (b) Boxplots of displacement and reaction force errors, both normalized by the standard deviation of the corresponding ground-truth, shown as a function of the training noise level $\stdNoise \in \{ 10^{-1},  10^{-2},  10^{-3},  10^{-4}, 0\}$, and of the observation type (full-fields vs boundary only, see legend). 
        Results are reported for the three material models considered (indicated at the top of each column) and for both the full-field and boundary observation cases (see the legend in the bottom-right panel).
        The three rows correspond to Setups~4--5--6, respectively. 
        Hence, while the first row shows the performance in the training setting (fitting regime), the second and third rows assess the ability of Neural-DFEM to generalize to mechanical setups different from those observed at training time.
        For Setup~6, no reaction force is defined, since no Dirichlet boundary conditions are prescribed.
        }
    \label{fig:results_summary_3D}
\end{figure}

For the IH material model, reducing the noise level in the training set leads to a saturation of the error, indicating that Neural-DFEM is unable to extract additional information from cleaner data.
This behavior is explained by the fact that the IH model is not polyconvex (as can be readily shown using \cite[Lemma~2.4]{hartmann2003polyconvexity}) and therefore cannot be approximated with arbitrary accuracy by a HNN, which is polyconvex by construction.
This behavior was not apparent in the plane-strain case, where the kinematics are restricted to a subset of deformations and the three invariants are no longer independent, effectively masking the non-polyconvex character of the IH model.
{Consequently, Neural-DFEM searches within the HNN hypothesis class for a polyconvex and thermodynamically admissible surrogate by minimizing the mismatch with the observed boundary responses.}
Importantly, despite this representational restriction, the identified surrogate still reproduces the macroscopic mechanical response with good accuracy (error of the order of $10^{-2}$).
In contrast, for the MR and FU material models, which are polyconvex, the approximation error decreases consistently as the noise level is reduced, {confirming increasingly accurate agreement with the underlying constitutive response} when it is not far from the considered hypothesis class.


Similarly to the two-dimensional case, the slightly larger error observed in Setup~6 compared to Setup~5 is explained by the fact that, as shown in Fig.~\ref{fig:test_cases_3D}d, the microscopic strain states in Setup~5 are closer to those of the training setup (Setup~4) than those in Setup~6.
Quantitatively, the Sinkhorn divergence from the empirical distribution of the principal stretches observed during training increases from Setup~5 to Setup~6: from 0.08 to 0.16 for IH, from 0.10 to 0.16 for MR, and from 0.08 to 0.15 for the FU model. 
This supports that generalization can be analyzed through the lens of the distributional shift in the strain distribution from the training to the test set.

\section{Discussion}
\label{sec:discussion}

This work shows that {hyperelastic isotropic} material model discovery can be successfully performed, at least in the considered synthetic test cases, even under partial observation regimes, relying solely on {dense measurements of} specimen boundary deformations and global reaction forces, and without paired strain--stress data. 
These results provide evidence that experimentally viable material discovery under partial observability may be feasible, although validation on real experimental data remains an important direction for future work.

The proposed framework is built upon a differentiable finite element solver embedded within the learning process. 
This design choice allows to establish a direct link between candidate constitutive laws and the resulting deformation fields, enabling a high flexibility in the formulation of the loss function, including, for instance, mean squared error terms evaluated on arbitrary subsets of the domain. 
This type of observation flexibility is not available in the same form in constitutive-learning approaches based on paired strain--stress data, nor in standard VFM-type formulations, where the measured displacement field is inserted directly into the equilibrium residual and must be sufficiently rich over the region where that residual is evaluated.
Moreover, because VFM relies on differentiating measured displacement fields -- an operation that amplifies measurement noise -- it typically necessitates a dedicated denoising stage. 
In contrast, the proposed approach avoids this requirement, promoting noise robustness.

Neural-DFEM relies on a neural-network-based constitutive ansatz, represented by the class of HNNs, which provides mathematically provable guarantees of physical and thermodynamic consistency, as well as {existence of solutions to} the associated equilibrium problem.
These properties are important in a solver-in-the-loop setting, because every intermediate constitutive model encountered during training must define a meaningful forward hyperelastic problem, not only the final learned model. 
They should, however, be distinguished from the practical convergence of the nonlinear finite element solver, which is addressed through the numerical strategies adopted in the training algorithm. 
In particular, quasi-Newton optimization, continuation, solver-aware backtracking, and the proposed initialization strategy jointly promote robustness and computational efficiency of the overall approach.
As a result, training can be carried out with limited computational resources. 
More precisely, all experiments were executed on a single CPU core of an AMD processor (up to 3.63 GHz), and training took on average 7.38 h (IH), 1.84 h (MR) and 1.38 h (FU) in the two-dimensional case; 
20.16 h (IH), 8.26 h (MR) and 24.79 h (FU) in the three-dimensional case.

An alternative paradigm that could, in principle, enable material discovery from partial observations is that of Physics-Informed Neural Networks (PINNs). 
However, our initial investigations (not reported here) revealed severe convergence difficulties for PINNs in this class of problems, even for comparatively simple configurations.
While further studies may be warranted, these observations suggest that approaches based on differentiable solvers are more effective for constitutive model discovery in mechanics.

In parallel with advances in material discovery, a growing body of work has focused on operator learning methods aimed at predicting the deformation of elastic bodies directly from data. 
Examples include approaches based on autoencoders \cite{kehls2025autoencoder} graph neural networks \cite{taghizadeh2025multifidelity}, PointNet architectures \cite{kashefi2023physics}, Universal Solution Manifold Networks \cite{zhang2025shape}, and DeepONets \cite{he2023novel}. 
While these methods can deliver extremely rapid predictions at inference time, they typically require large training datasets and often struggle to generalize to geometries and boundary conditions that differ substantially from those encountered during training. 
As such, these methods primarily function as surrogate models, acting as accelerators of numerical solvers trained on extensive sets of precomputed simulations that provide the dominant learning bias.
By contrast, the considered framework relies on strong inductive biases rooted in fundamental physical principles: the momentum balance equation is enforced explicitly at every training epoch through the finite element solver, and the constitutive model is constrained by requirements such as frame indifference, polyconvexity, and coercivity. 
These ingredients reduce the dependence on purely data-driven learning bias and promote generalization across geometries and boundary conditions well beyond those used during training, provided that they induce strain states sufficiently represented during training, as illustrated by the numerical experiments.
From this perspective, the learning task shifts from approximating the global, macroscopic response of a body to identifying its local, microscopic material behavior.

Consequently, within this framework, concepts such as generalization and extrapolation are most naturally interpreted at the level of local strain states. 
In particular, generalization across experimental setups can be understood as a form of distributional shift in the space of strain invariants. 
Our results indicate that predictive accuracy deteriorates as the strain-state distribution encountered at test time departs from that sampled during training, consistently with the broader machine-learning literature on out-of-distribution generalization.
This behavior is not specific to the present method but reflects a fundamental limitation of data-driven material discovery: when local deformation states are insufficiently represented in the training data, the inferred surrogate is constrained only by the inductive biases embedded in the hypothesis class, but too strong structural priors typically sacrifice expressivity.
{More in general, for a finite set of experiments, different constitutive responses may be compatible with the available measurements. 
Neural-DFEM therefore identifies an admissible hyperelastic material law consistent with the sampled observations and strain states, rather than guaranteeing unique recovery of the underlying strain-energy density over its entire domain.}
Therefore, analyzing the statistical distribution of local strain states and quantifying their distributional mismatch relative to the training regime provides a physically grounded lens through which generalization performance can be interpreted. 
Such analysis may serve as an \emph{a posteriori} diagnostic of prediction reliability and suggests a potential guideline for designing informative training configurations.

{An inherent limitation of the proposed approach is the trade-off between structural guarantees and expressivity of the HNN hypothesis class: since HNNs are restricted to polyconvex strain-energy densities, genuinely non-polyconvex constitutive responses cannot, in general, be represented with arbitrary accuracy, as illustrated by the Ishihara example in three dimensions.}
Another limitation of the present work is its restriction to isotropic materials. 
Nevertheless, the proposed HNN framework can be extended in a straightforward manner to anisotropic settings by incorporating anisotropic invariants into the model input. 
If the directions of anisotropy are known \emph{a priori}, the extension is immediate; if they are unknown, they may be introduced as additional variables to be inferred within the optimization process.
{The present framework is also restricted to path-independent hyperelasticity; extensions to history-dependent constitutive behavior would require substantial modifications to the current formulation.}
A further limitation of the present study is that the numerical assessment is restricted to synthetic data generated from known constitutive models. 
{In addition, data generation and identification employ the same finite-element discretization and solver.}
While this controlled setting enables a rigorous quantitative evaluation of generalization accuracy and robustness to noise {without confounding discretization effects, the reported absolute errors may be optimistic with respect to settings involving discretization mismatch. Validation on experimental measurements remains therefore} an important direction for future work. 
In particular, applying the proposed framework to displacement fields acquired through DIC or related experimental techniques will allow the impact of model discrepancies, measurement artifacts, and experimental uncertainties to be assessed.

Ultimately, this study highlights the importance of incorporating fundamental physical principles as inductive biases directly into machine learning models for scientific discovery in computational mechanics.
Constraining the hypothesis space to thermodynamically consistent {and mathematically admissible constitutive laws, satisfying sufficient conditions for the existence of equilibrium solutions,} and enforcing equilibrium throughout training structurally regularizes the inverse problem, enabling reliable generalization even under partial observability.

\section*{Acknowledgements}
The present research has received support from the project FIS, MUR, Italy 2025-2028, Project code: FIS-2023-02228, CUP: D53C24005440001, ``SYNERGIZE: Synergizing Numerical Methods and Machine Learning for a new generation of computational models''.
The author acknowledges MUR, grant Dipartimento di Eccellenza 2023-2027.
The author is member of GNCS, ``Gruppo Nazionale per il Calcolo Scientifico'' (National Group for Scientific Computing) of INdAM (Istituto Nazionale di Alta Matematica).

\section*{Source code and data}

The source code supporting the findings of this study is publicly available in the Neural-DFEM GitHub repository at \textcolor{blue}{\url{https://github.com/FrancescoRegazzoni/Neural-DFEM}}.

The repository includes the implementation of the proposed method and neural constitutive architecture, together with the scripts and configuration files required to reproduce the numerical experiments reported in this work.

\appendix
\section*{Appendix}
\addcontentsline{toc}{section}{Appendix}

\section{Material models} \label{sec:material_models}

We consider the following ground-truth constitutive laws.

\begin{itemize}
    \item \textbf{Ishihara (IH)} \cite{ishihara1948statistical}, 
    with $C_1 = 0.5$, $C_2 = 1$, $C_3 = 3$, $K = 1.5$:
    \begin{equation}
        \W(\F) = C_1(\isoinvI-3) + C_2(\isoinvII-3) + C_3(\isoinvI-3)^2 + K(\J-1)^2;
    \end{equation}

    \item \textbf{Mooney-Rivlin (MR)} \cite{mooney1940theory,rivlin1948large},
    with $C_1 = 1.0$, $C_2 = 0.8$, $K = 1$:
    \begin{equation}
        \W(\F) = C_1 (\invI - 3 - 2\log\J) + C_2 (\invII - 3 - 4\log\J) + \frac{1}{2}K(\J-1)\log\J;
    \end{equation}
    
    \item \textbf{Fung (FU)} \cite{fung2013biomechanics},
    with $C = 1$, $b = 3$, $K = 1.5$:
    \begin{equation}
        \W(\F) = \frac{C}{2b} \left(e^{b (\isoinvI - 3)} - 1\right) + \frac{K}{4} ((J-1)^2 + \log(J)^2).
    \end{equation}
\end{itemize}

\section{Mechanical setups of test cases} \label{sec:test_cases}

In this section, we provide details regarding the test cases considered in this work.

\paragraph{Setup~1.} 
Setup~1, taken from \cite{flaschel2021unsupervised}, consists of a square plate with edge length 2 containing a central circular hole of radius 0.1, as shown in Fig.~\ref{fig:test_cases_2D}a.
The mesh consists of 1441 nodes.
The plate is subjected to asymmetric, displacement-controlled biaxial tension, with prescribed displacements of $\delta$ in the vertical direction and $\delta/2$ in the horizontal direction, with $\delta \in \{ 0.1, 0.2, \dots, 0.8\}$.
For computational convenience, only the top-left quarter of the plate is simulated, with symmetry boundary conditions applied to account for the remaining portion of the domain.
Specifically, the following boundary conditions are applied:
\begin{equation} 
    \begin{aligned}
        &\displ \cdot \normal = 0, \, &&(\Id - \normal\otimes\normal)(\Piola\normal) = \mathbf{0} 
        && \qquad \text{on } \Gamma_{\text{left}} \cup \Gamma_{\text{down}}, 
        \\
        &\displ \cdot \normal = \delta, \, &&(\Id - \normal\otimes\normal)(\Piola\normal) = \mathbf{0} 
        && \qquad \text{on } \Gamma_{\text{up}}, 
        \\
        &\displ \cdot \normal = \delta/2, \, &&(\Id - \normal\otimes\normal)(\Piola\normal) = \mathbf{0} 
        && \qquad \text{on } \Gamma_{\text{right}}, 
        \\
        &\Piola \normal = \mathbf{0}  &&
        && \qquad \text{on } \partial\Omega \setminus (\Gamma_{\text{down}} \cup \Gamma_{\text{right}} \cup \Gamma_{\text{up}} \cup \Gamma_{\text{left}}).\\
    \end{aligned}
\end{equation}

\paragraph{Setup~2.} 

Setup~2, also from \cite{flaschel2021unsupervised}, consists of square plate with two ellipsoidal holes, as shown in Fig.~\ref{fig:test_cases_2D}b. The mesh has 4334 nodes. The specimen is subjected to displacement-controlled uniaxial tension in the vertical direction, with displacement $\delta \in \{0.1, 0.2, \dots, 1.0\}$, while the remaining portions of the boundary sides (including those of the inner holes) are traction-free.
The boundary conditions read as follows:
\begin{equation} 
    \begin{aligned}
        &\displ = \mathbf{0}
        && \qquad \text{on } \Gamma_{\text{down}}, 
        \\
        &\displ = 
        \begin{pmatrix} 0 \\ \delta \end{pmatrix} 
        && \qquad \text{on } \Gamma_{\text{up}}, 
        \\
        &\Piola \normal = \mathbf{0} 
        && \qquad \text{on } \partial\Omega \setminus (\Gamma_{\text{down}} \cup \Gamma_{\text{up}}).\\
    \end{aligned}
\end{equation}

\paragraph{Setup~3.} 

Setup~3, adapted from \cite{zhang2025shape}, consists of 10 distinct domains, each defined as a square plate of unit edge length containing a random number of circular holes -- ranging from 1 to 3 -- with randomly sampled sizes and positions.
Some samples are displayed in Fig.~\ref{fig:test_cases_2D}c. 
The number of mesh nodes per each geometry ranges between 5300 and 6000.
The plates are subjected to normal-spring boundary conditions (with constant $k_{\perp} = 0.01$) and loaded on the two vertical sides by a uniformly distributed traction $F \in \{-0.1, 0.1, 0.2, 0.3\}$, where negative values correspond to compressive loading and positive values to tensile loading.
More precisely, the boundary conditions are given by:
\begin{equation} 
    \begin{aligned}
        &\Piola \normal + k_{\perp} (\normal\otimes\normal) \displ = \mathbf{0} 
        && \qquad \text{on } \Gamma_{\text{top}} \cup \Gamma_{\text{down}}, 
        \\
        &\Piola \normal + k_{\perp} (\normal\otimes\normal) \displ = F \normal 
        && \qquad \text{on } \Gamma_{\text{left}} \cup \Gamma_{\text{right}}, 
        \\
        &\Piola \normal = \mathbf{0} 
        && \qquad \text{on } \partial\Omega \setminus (\Gamma_{\text{down}} \cup \Gamma_{\text{right}} \cup \Gamma_{\text{up}} \cup \Gamma_{\text{left}}).\\
    \end{aligned}
\end{equation}

\paragraph{Setup 4.}
Setup 4 (see Fig.~\ref{fig:test_cases_3D}a) represents a three-dimensional generalization of the two-dimensional Setup~1.
It consists of a cube with an edge length 2 and a circular hole with radius 0.5.
The mesh consists of 7093 nodes.
Similarly to Setup~1, for symmetry reasons we simulate only 1/8 of the domain.
The specimen is subjected to asymmetric, displacement-controlled triaxial tension, with prescribed displacements of $\delta$, $\delta/2$ and $\delta/4$ in the three orthogonal directions, with $\delta \in \{ 0.1, 0.2, \dots, 0.5\}$.
The boundary conditions hence read as follows:
\begin{equation} 
    \begin{aligned}
        &\displ \cdot \normal = 0, \, &&(\Id - \normal\otimes\normal)(\Piola\normal) = \mathbf{0} 
        && \qquad \text{on } \Gamma_{\text{left}} \cup \Gamma_{\text{down}} \cup \Gamma_{\text{front}}, 
        \\
        &\displ \cdot \normal = \delta, \, &&(\Id - \normal\otimes\normal)(\Piola\normal) = \mathbf{0} 
        && \qquad \text{on } \Gamma_{\text{right}}, 
        \\
        &\displ \cdot \normal = \delta/2, \, &&(\Id - \normal\otimes\normal)(\Piola\normal) = \mathbf{0} 
        && \qquad \text{on } \Gamma_{\text{back}}, 
        \\
        &\displ \cdot \normal = \delta/4, \, &&(\Id - \normal\otimes\normal)(\Piola\normal) = \mathbf{0} 
        && \qquad \text{on } \Gamma_{\text{up}}, 
        \\
        &\Piola \normal = \mathbf{0}  &&
        && \qquad \text{on } \partial\Omega \setminus (\Gamma_{\text{up}} \cup \Gamma_{\text{down}} \cup \Gamma_{\text{left}} \cup \Gamma_{\text{right}} \cup \Gamma_{\text{front}} \cup \Gamma_{\text{back}}).\\
    \end{aligned}
\end{equation}

\paragraph{Setup 5.}
Setup~5 consists of a cubic specimen featuring an oblique through-hole, as illustrated in Fig.~\ref{fig:test_cases_3D}b, based on a 2585-node mesh.
One face of the specimen is fully constrained by imposing zero displacement, while the opposite face is subjected to a displacement-controlled load combining two components: a uniaxial traction of magnitude $\delta$ and a torsional deformation with rotation angle $\theta = \tfrac{2\pi}{5}\,\delta$, where $\delta \in \{0.1, 0.2, \dots, 0.5\}$.
Specifically, the applied boundary conditions are defined as follows, where $\mathbf{x} = (x,y,z)^\mathsf{T}$ denotes the spatial coordinates, with the $z$-axis oriented along the vertical direction:
\begin{equation} 
    \begin{aligned}
        &\displ = \mathbf{0}
        && \qquad \text{on } \Gamma_{\text{down}}, 
        \\
        &\displ = 
        \begin{pmatrix} (\cos\theta-1)x-\sin\theta \, y \\ \sin\theta \, x+(\cos\theta-1)y \\ \delta \end{pmatrix}
        && \qquad \text{on } \Gamma_{\text{up}}, 
        \\
        &\Piola \normal = \mathbf{0} 
        && \qquad \text{on } \partial\Omega \setminus (\Gamma_{\text{down}} \cup \Gamma_{\text{up}}).\\
    \end{aligned}
\end{equation}

\paragraph{Setup 6.}
Setup 6 involves a more complex geometry than the previous cases, based on a three-dimensional mounting clip from a publicly available open-source CAD model~\cite{thingiverse_clip}. 
This setup is designed to assess the generalization capability of the proposed method to markedly different geometries and mixed boundary conditions, representative of nontrivial, real-world-inspired configurations.
The mesh features 12729 nodes.
The geometry, shown in Fig.~\ref{fig:test_cases_3D}c, is constrained by imposing zero vertical displacement on the bottom surface, corresponding to a roller-type (sliding) boundary condition. The two arms of the clip are actuated by a follower load of magnitude $F \in \{0.01, 0.02, \dots, 0.06\}$, which tends to open them, as illustrated in the figure, while a uniformly distributed forward load (i.e. in the direction $\mathbf{e}_{\text{front}}$) of magnitude~$F/5$ is simultaneously applied.
In addition, the body is subjected to normal-spring boundary conditions on both the front surface, with spring constant $k_{\text{front}} = 0.01$, and the surface of the central hole, with spring constant $k_{\text{hole}} = 0.1$. These elastic supports are introduced to induce nontrivial internal stress states when the clip is pushed forward.
To summarize, the boundary conditions are as follows:
\begin{equation} 
    \begin{aligned}
        &\displ \cdot \normal = 0, \quad(\Id - \normal\otimes\normal)(\Piola\normal) = \mathbf{0} 
        && \qquad \text{on } \Gamma_{\text{down}}, 
        \\
        &\Piola \normal + k_{\text{front}} (\normal\otimes\normal) \displ = \mathbf{0} 
        && \qquad \text{on } \Gamma_{\text{front}}, 
        \\
        &\Piola \normal + k_{\text{hole}} (\normal\otimes\normal) \displ = \mathbf{0} 
        && \qquad \text{on } \Gamma_{\text{hole}}, 
        \\
        &\Piola \normal = - F \J \F^{-T} \normal - F/5 \, \mathbf{e}_{\text{front}}
        && \qquad \text{on } \Gamma_{\text{back}}, 
        \\
        &\Piola \normal = \mathbf{0}
        && \qquad \text{on } \partial\Omega \setminus (\Gamma_{\text{up}} \cup \Gamma_{\text{down}} \cup \Gamma_{\text{front}} \cup \Gamma_{\text{hole}} \cup \Gamma_{\text{back}}).\\
    \end{aligned}
\end{equation}

\section{Error metrics}
\label{sec:metrics}

Given a collection of $N$ points, and associated displacement vectors $\mathbf{d}_i$, for $i = 1,\dots,N$, we define the average displacement as
\begin{equation}
    \overline{\mathbf{d}} = \frac{1}{N} \sum_{i=1}^N \mathbf{d}_i,
\end{equation}
and the displacement variance as 
\begin{equation}
    \Var(\mathbf{d}) = \frac{1}{N} \sum_{i=1}^N \| \mathbf{d}_i - \overline{\mathbf{d}}\|^2.
\end{equation}
Now, let us consider a collection of reference displacements $\hat{\mathbf{d}}_i$, for $i = 1,\dots,N$. We define the \textit{root mean square error} (RMSE) between $\{\mathbf{d}_i\}_i$ and $\{\hat{\mathbf{d}}_i\}_i$ as: 
\begin{equation}
    \text{RMSE} = \sqrt{\frac{1}{N} \sum_{i=1}^N \| \mathbf{d}_i - \hat{\mathbf{d}}_i\|^2}.
\end{equation}
Consequently, we define the \textit{variance-normalized root mean square error} (vRMSE) as:
\begin{equation}
    \text{vRMSE} = \frac{\text{RMSE}}{\sqrt{\Var(\hat{\mathbf{d}})}}.
\end{equation}
This definition is such that a naïve predictor, predicting the average displacement in each point, would score vRMSE = 1, while a perfect predictor achieving zero error satisfies vRMSE = 0.


\section{Canonical deformations}
\label{sec:canonical_deformations}

In Fig.~\ref{fig:test_cases_2D}c, in addition to the distribution of principal stretches corresponding to the considered mechanical setups, the principal stretches associated with the following canonical deformations are shown for reference:
\begin{itemize}
    \item Uniaxial tension: $\F = \begin{pmatrix} 1+\delta & 0 \\ 0 & 1 \end{pmatrix}$;
    \item Uniaxial compression: $\F = \begin{pmatrix} \frac{1}{1+\delta} & 0 \\ 0 & 1 \end{pmatrix}$;
    \item Biaxial tension: $\F = \begin{pmatrix} 1+\delta & 0 \\ 0 & 1+\delta \end{pmatrix}$;
    \item Biaxial compression: $\F = \begin{pmatrix} \frac{1}{1+\delta} & 0 \\ 0 & \frac{1}{1+\delta} \end{pmatrix}$;
    \item Simple shear: $\F = \begin{pmatrix} 1 & \delta \\ 0 & 1 \end{pmatrix}$.
\end{itemize}
In all cases, the range $\delta \in [0, 0.5]$ is represented in the figure.


\section{Sensitivities computation} \label{sec:sensitivities_computation}

In this section, we describe how sensitivities of the loss function with respect to the trainable parameters are computed by combining adjoint-based differentiation through the finite element equilibrium problem with automatic differentiation.

To this aim, it is useful to write the Neural-DFEM formulation in fully discrete form. For the $i$-th experiment, let us associate with the finite element displacement field $\displ_h^i$ (see \cref{sec:FEM_formulation}) the vector of free degrees of freedom
\begin{equation}
    \displDOF^i = (d_1^i,d_2^i,\ldots,d_{\numDOF^i}^i)^T
    \in \mathbb{R}^{\numDOF^i}.
\end{equation}
Furthermore, let $\NNparams \in \mathbb{R}^{\NNnumparams}$ denote the vector collecting all trainable parameters of the HNN. The discrete equilibrium problem associated with the $i$-th experiment then consists in finding $\displDOF^i \in \mathbb{R}^{\numDOF^i}$ such that
\begin{equation} \label{eqn:equilibrium_discrete}
    \DiscreteBilinearResidual^i(\displDOF^i;\NNparams) = \mathbf{0},
\end{equation}
where the components of the discrete residual are defined by
\begin{equation}
    \DiscreteBilinearResidual^i(\displDOF^i;\NNparams)
    \cdot \mathbf{e}_n
    =
    \bilinearResidual^i(\displ_h^i,\basisfun{n}^i),
    \qquad n=1,\ldots,\numDOF^i.
\end{equation}
Here and in the following, the correspondence between $\displDOF^i$ and the associated finite element function $\displ_h^i$, as well as that between the trainable parameters $\NNparams$ and the corresponding material law $\W$, is left implicit.

Using this notation, we define the objective functional associated with the $i$-th experiment as
\begin{equation}
    \lossDispl^i(\displDOF^i,\NNparams)
    =
    \sum_{j=1}^{\numOBSpts{i}}
    \|
        \displOBS{i}{j}
        -
        \displ_h^i(\ptsOBS{i}{j})
    \|^2
    +
    \alpha_R
    \sum_{k=1}^{\numOBSreaction{i}}
    \|
        \reactionOBS{i}{k}
        -
        \reaction{i}{k}(\displ_h^i)
    \|^2 .
\end{equation}
For a given set of trainable parameters $\NNparams$, let $\displDOF^i(\NNparams)$ denote the solution of \eqref{eqn:equilibrium_discrete}. We can then introduce the reduced functional
\begin{equation}
    \lossParam^i(\NNparams)
    =
    \lossDispl^i\bigl(\displDOF^i(\NNparams),\NNparams\bigr),
\end{equation}
in which the displacement degrees of freedom have been eliminated through the equilibrium constraint.

The training problem can therefore be written as
\begin{equation}
    \NNparams^\ast
    =
    \argmin{\NNparams\in\mathbb{R}^{\NNnumparams}}
    \sum_{i=1}^{\numOBS}
    \lossDispl^i(\displDOF^i,\NNparams)
    \qquad
    \text{subject to }
    \DiscreteBilinearResidual^i(\displDOF^i;\NNparams)=\mathbf{0},
    \quad i=1,\ldots,\numOBS,
\end{equation}
or, equivalently, in reduced form,
\begin{equation}
    \NNparams^\ast
    =
    \argmin{\NNparams\in\mathbb{R}^{\NNnumparams}}
    \sum_{i=1}^{\numOBS}
    \lossParam^i(\NNparams).
\end{equation}
Training requires the gradient of the reduced loss with respect to the HNN parameters. Rather than differentiating through the individual iterations of the nonlinear finite element solver, we account for the implicit dependence of $\displDOF^i$ on $\NNparams$ through an algebraic adjoint formulation \cite{farrell2013automated}. Specifically, for each experiment we compute the adjoint variable
$\displDOFadjoint^i\in\mathbb{R}^{\numDOF^i}$ by solving
\begin{equation}
    \label{eqn:adjoint_system}
    \left(
        \frac{\partial
        \DiscreteBilinearResidual^i}
        {\partial \displDOF^i}
    \right)^T
    \displDOFadjoint^i
    =
    \nabla_{\displDOF^i}
    \lossDispl^i(\displDOF^i,\NNparams).
\end{equation}
All quantities in \eqref{eqn:adjoint_system} are evaluated at the equilibrium solution
$\displDOF^i=\displDOF^i(\NNparams)$. The gradient of the reduced functional is then given by
\begin{equation}
    \label{eqn:sensitivity_computation}
    \nabla_{\NNparams}\lossParam^i(\NNparams)
    =
    -
    \left(
        \frac{\partial
        \DiscreteBilinearResidual^i}
        {\partial \NNparams}
    \right)^T
    \displDOFadjoint^i
    +
    \nabla_{\NNparams}
    \lossDispl^i(\displDOF^i,\NNparams).
\end{equation}
The gradient of the total objective is finally obtained by summing the contributions of all experiments:
\begin{equation}
    \nabla_{\NNparams}
    \sum_{i=1}^{\numOBS}\lossParam^i(\NNparams)
    =
    \sum_{i=1}^{\numOBS}
    \nabla_{\NNparams}\lossParam^i(\NNparams).
\end{equation}
Automatic differentiation is used to evaluate the explicit derivatives of the HNN-dependent quantities entering the discrete residual and the loss function, whereas the implicit dependence of the equilibrium displacement on the trainable parameters is handled through the adjoint problem \eqref{eqn:adjoint_system}. In particular, the Newton iterations used to solve the nonlinear equilibrium problem are not unrolled and differentiated individually.

From a computational viewpoint, a complete loss-and-gradient evaluation consists of the following steps:
\begin{itemize}
    \item For each experiment, the equilibrium problem \eqref{eqn:equilibrium_discrete} is solved using Newton's method, warm-started from the equilibrium solution obtained at the previous optimization iteration. If $n_{\mathrm{Newt}}^i$ Newton iterations are required for experiment $i$, this step involves $n_{\mathrm{Newt}}^i$ sparse linear solves of dimension $\numDOF^i$.

    \item Once equilibrium has been reached, the adjoint problem \eqref{eqn:adjoint_system} is solved. This requires one additional sparse linear solve of dimension $\numDOF^i$, whose system matrix is the transpose of the equilibrium Jacobian evaluated at the converged solution. Formula \eqref{eqn:sensitivity_computation} is then used to obtain the gradient with respect to all $\NNnumparams$ trainable parameters. Consequently, the number of additional finite element linear solves required by the adjoint gradient computation is independent of the number of trainable parameters.

    \item The parameters $\NNparams$ are updated according to the selected quasi-Newton optimization method. When line search is employed, trial parameter values may require additional forward equilibrium solves to evaluate the loss. An additional adjoint solve is required whenever the gradient is also evaluated at a trial point.
\end{itemize}

Hence, disregarding additional line-search evaluations, one complete evaluation of the objective and its gradient requires
\begin{equation}
    \sum_{i=1}^{\numOBS}
    \left(n_{\mathrm{Newt}}^i+1\right)
\end{equation}
sparse finite element linear solves. The additional ``$+1$'' term corresponds to the adjoint problem and represents the main computational overhead associated with differentiation through the equilibrium constraint.


\section{HNN initialization} 
\label{sec:initialization_proof}

In this section, we analytically prove that, under the proposed parametrization and initialization strategy, the derivative of a HNN with respect to the invariants $\invI$ and $\invII$ at the reference configuration remains bounded as the number of neurons increases. 
For brevity, we present the derivation for $\invI$ only, as the same arguments apply to $\invII$.

\paragraph{Case with skip connections.}

Consider first the architecture including skip connections. 
From \eqref{eqn:Wnn}, we obtain
\begin{equation} \label{eqn:dWnn_dI1}
    \left\{
    \begin{aligned}
        \frac{\partial\neur_{1}}{\partial \invI} 
        & = \weightI{1} 
            \odot \act'(      \weightI{1}(\invI - 3) 
                            + \weightII{1}(\invII - 3)
                            + \weightJ{1}(\J - 1) 
                            + \bias{1}
                            ), &\\
        \frac{\partial\neur_{i}}{\partial \invI} 
        & = \left(\weightI{i} + \weightZ{i-1} \frac{\partial\neur_{i-1}}{\partial \invI} \right)
            \odot \act'(      \weightI{i}(\invI - 3) 
                            + \weightII{i}(\invII - 3) 
                            + \weightJ{i}(\J - 1) 
                            + \weightZ{i-1} \neur_{i-1}
                            + \bias{i} 
                            ), &\quad\text{for } i = 2, \ldots, \numLay\\
        \frac{\partial\neurOut}{\partial \invI} & = \Wconst \, \weightOUT \cdot \frac{\partial\neur_{\numLay}}{\partial \invI},&
    \end{aligned}
    \right.
\end{equation}
where $\odot$ denotes the Hadamard (i.e., componentwise) product. 
In the reference configuration ($\F=\Id$), all pre-activations vanish, hence $\act'(\cdot)=\act'(0)$. Therefore,
\begin{equation} \label{eqn:dWnn_dI1_Id}
    \left\{
    \begin{aligned}
        \frac{\partial\neur_{1}}{\partial \invI} 
        & = \weightI{1}  \act'(0), &\\
        \frac{\partial\neur_{i}}{\partial \invI} 
        & = \left(\weightI{i} + \weightZ{i-1} \frac{\partial\neur_{i-1}}{\partial \invI} \right) \act'(0), &\quad\text{for } i = 2, \ldots, \numLay\\
        \frac{\partial\neurOut}{\partial \invI} & = \Wconst \, \weightOUT \cdot \frac{\partial\neur_{\numLay}}{\partial \invI}.&
    \end{aligned}
    \right.
\end{equation}
Taking expectations with respect to the random initialization (leveraging independence and identical distributions of the trainable parameters), and denoting $\Exp{\pos(\stdInit Z)} = \expW$, where $Z \sim \mathcal{N}(0,1)$ is a standard Gaussian random variable, we obtain the recurrence:
\begin{equation}
    \left\{
    \begin{aligned}
        \Exp{\frac{\partial\neur_{1}}{\partial \invI}}
        & = \expW \act'(0) \ones{\numNeur{1}}, &\\
        \Exp{\frac{\partial\neur_{i}}{\partial \invI}}
        & = \expW \act'(0) \left(\ones{\numNeur{i}} + \frac{1}{\numNeur{i-1}} \ones{\numNeur{i}\times\numNeur{i-1}} \Exp{\frac{\partial\neur_{i-1}}{\partial \invI}} \right) , &\quad\text{for } i = 2, \ldots, \numLay\\
        \Exp{\frac{\partial\neurOut}{\partial \invI}} 
        & = \Wconst \, \expW \, \frac{1}{\numNeur{\numLay}} \ones{\numNeur{\numLay}} \cdot \Exp{\frac{\partial\neur_{\numLay}}{\partial \invI}}.&
    \end{aligned}
    \right.
\end{equation}
By defining the constant $\gamma = \expW \act'(0)$, we have by induction:
\begin{equation}
    \begin{aligned}
    \Exp{\frac{\partial\neur_{2}}{\partial \invI}} &= \gamma \left(1 + \gamma \right) \ones{\numNeur{2}} \\
    \Exp{\frac{\partial\neur_{3}}{\partial \invI}} &= \gamma \left(1 + \gamma \left(1 + \gamma \right) \right) \ones{\numNeur{3}} \\
    \vdots\quad\; & \\
    \Exp{\frac{\partial\neur_{i}}{\partial \invI}} &= \sum_{k = 1}^i \left(\gamma \right)^k \ones{\numNeur{i}} 
    = \frac{\gamma}{1-\gamma} \left(1 - \gamma^i\right) \ones{\numNeur{i}} \\
    \end{aligned}
\end{equation}
where the last equation holds for any $i = 1,\dots,\numLay$.
In the above calculation, we have exploited the identity $\frac{1}{\numNeur{i-1}} \ones{\numNeur{i}\times\numNeur{i-1}} \ones{\numNeur{i-1}} = \ones{\numNeur{i}}$.
Finally, the derivative of the output neuron with respect to $\invI$ reads:
\begin{equation}\label{eqn:exp_dWnn_dI1_skip-conn}
\Exp{\frac{\partial\neurOut}{\partial \invI}} 
        = \Wconst \, \expW \, 
        \frac{\gamma}{1-\gamma} \left(1 - \gamma^\numLay\right) 
\end{equation}
This means that the expected value of the derivative with respect to $\invI$ is independent of the number of neurons of each layer.
It instead depends on the number of layers.
However, Eq.~\eqref{eqn:exp_dWnn_dI1_skip-conn} provides a guideline to avoid blows up of this quantity when deep architecture are employed: indeed, $\stdInit$ and $\act$ must be chosen so that the constant $\gamma = \expW \act'(0)$ is smaller than one. In this case, indeed, we have a finite limit:
\begin{equation}
\Exp{\frac{\partial\neurOut}{\partial \invI}} 
        \xrightarrow[L \to +\infty]{}
        \Wconst \, \expW \, \frac{\gamma}{1-\gamma} 
\end{equation}

\paragraph{Case without skip connections.}

Let us now consider the case when skip connections are not employed, that is when the terms $\weightI{i}$, $\weightII{i}$ and $\weightJ{i}$ for $i \geq 2$ are removed.
By proceeding similarly to above, we have:
\begin{equation}
    \left\{
    \begin{aligned}
        \Exp{\frac{\partial\neur_{1}}{\partial \invI}}
        & = \gamma \ones{\numNeur{1}}, &\\
        \Exp{\frac{\partial\neur_{i}}{\partial \invI}}
        & = \gamma \frac{1}{\numNeur{i-1}} \ones{\numNeur{i}\times\numNeur{i-1}} \Exp{\frac{\partial\neur_{i-1}}{\partial \invI}} , &\quad\text{for } i = 2, \ldots, \numLay\\
        \Exp{\frac{\partial\neurOut}{\partial \invI}} 
        & = \Wconst \, \expW \, \frac{1}{\numNeur{\numLay}} \ones{\numNeur{\numLay}} \cdot \Exp{\frac{\partial\neur_{\numLay}}{\partial \invI}}.&
    \end{aligned}
    \right.
\end{equation}
This entails that
\begin{equation}
    \Exp{\frac{\partial\neur_{i}}{\partial \invI}} = \gamma^i \ones{\numNeur{i}},
\end{equation}
and, in conclusion
\begin{equation}\label{eqn:exp_dWnn_dI1_no-skip-conn}
\Exp{\frac{\partial\neurOut}{\partial \invI}} 
= \Wconst \, \expW \, \gamma^\numLay.
\end{equation}
In this case, the expected derivative is again independent of the number of neurons per layer. 
For deep architectures, choosing $\gamma$ close to one prevents both vanishing and blow-up of the gradient.


\section{Hyperparameters tuning}
\label{sec:hyperparameters_tuning}

In this work, the following hyperparameters of the HNNs are tuned via a grid search over finite sets of candidate values, listed below:

\begin{itemize}
    \item the number of layers, $\numLay \in \{1,2,3\}$;
    \item the number of neurons in each layer, $\numNeur{i} \in \{5, 10, 20\}$;
    \item the presence or absence of skip connections;
    \item the choice of input invariants, either the native set $(\invI, \invII, \J)$ or the isochoric set $(\isoinvI, \isoinvII^{3/2}, \J)$, where the exponent $3/2$ ensures polyconvexity, as discussed in \cref{rem:isochoric_invariant_input};
    \item the initialization standard deviation, $\stdInit \in \{0.05, 0.1, 0.2, 0.4, 0.6, 0.8\}$;
    \item the dimensional scaling constant, $\Wconst \in \{1, 5, 10, 20\}$.
\end{itemize}

The optimal configuration is selected as the one achieving, for $\stdNoise = 10^{-3}$, the lowest training loss.
Owing to the strong structural priors embedded in the HNN architecture, which substantially restrict the admissible hypothesis space, we do not employ cross-validation for hyperparameters tuning, as discussed in the Results.
The selected hyperparameters values are reported in Tab.~\ref{tab:hyperparameters}.

\begin{table}[ht]
\centering
\begin{tabular}{lcccccc}
\toprule
 & \multicolumn{3}{c}{2D specimens} 
 & \multicolumn{3}{c}{3D specimens} \\
\cmidrule(lr){2-4} \cmidrule(lr){5-7}
Hyperparameter 
 & IH & MR & FU 
 & IH & MR & FU \\
\midrule
Number of layers $\numLay$ 
 & 2  & 1 & 1 & 1 & 1 & 2 \\

Number of neurons per layer $\numNeur{i}$ 
 & 5 & 5 & 5 & 5 & 5 & 5 \\

Skip connections 
 & \ding{55} & - & - & - & - & \ding{55} \\

Isochoric input invariants 
 & \ding{55} & \ding{55} & \ding{55} & \ding{55} & \ding{55} & \ding{51} \\

Initialization std. $\stdInit$ 
 & 0.6 & 0.1 & 0.1 & 0.1 & 0.1 & 0.6 \\

Scaling constant $\Wconst$ 
 & 1 & 10 & 10 & 10 & 10 & 1 \\
\bottomrule
\end{tabular}
\caption{Selected hyperparameters for the HNN models.}
\label{tab:hyperparameters}
\end{table}


\section{Basic lemmas on convex functions}

In this section we collect standard definitions and results an convex functions, employed in this work.

\begin{definition}
    A function $f\colon \mathbb{R}^n \to \mathbb{R}$ is convex if for all $\mathbf{x}, \mathbf{y} \in \mathbb{R}^n$ and $\lambda \in [0,1]$:
    \begin{equation} \label{eqn:convexity_definition}
    f(\lambda \mathbf{x} + (1-\lambda) \mathbf{y}) \leq \lambda f(\mathbf{x}) + (1-\lambda) f(\mathbf{y}).
    \end{equation}
\end{definition}

\begin{proposition}\label{prop:convex_nondecreasing_composition}
    Consider a convex function $f\colon \mathbb{R}^n \to \mathbb{R}$ and a convex, non-decreasing function $g \colon \mathbb{R} \to \mathbb{R}$. Then the composition $g \circ f \colon \mathbb{R}^n \to \mathbb{R}$ is also convex.
\end{proposition}
\begin{proof}
    Let $\mathbf{x}, \mathbf{y} \in \mathbb{R}^n$ and $\lambda \in [0,1]$. Since $g$ is non-decreasing, we can apply it to both sides of \eqref{eqn:convexity_definition}:
    \begin{equation}
        \begin{split}
            g(f(\lambda \mathbf{x} + (1-\lambda) \mathbf{y})) 
            &\leq g(\lambda f(\mathbf{x}) + (1-\lambda) f(\mathbf{y}))\\
            &\leq \lambda g(f(\mathbf{x})) + (1-\lambda) g(f(\mathbf{y})),
        \end{split}
    \end{equation}
    where in the last line we have used the convexity of $g$. This proves that $g \circ f$ is convex.
\end{proof}

\begin{lemma}\label{prop:convex_nondecreasing_nonconst}
Let $f \colon \mathbb{R} \to \mathbb{R}$ be a $C^1$ function which is convex, non-decreasing, and non-constant. Then there exist constants $a>0$ and $b\in\mathbb{R}$ such that
\[
    f(x) \;\geq\; a x - b, \qquad \forall x \in \mathbb{R}.
\]
\end{lemma}

\begin{proof} 
Since $f$ is $C^1$, non-decreasing but non-constant, then there exists $x_0 \in \mathbb{R}$ such that $f'(x_0) > 0$. 
Moreover, as $f$ is convex, we have
\begin{equation}
    f(x) \geq f(x_0) + f'(x_0)(x-x_0), \qquad \forall \, x \in \mathbb{R},
\end{equation}
which proves the thesis.
\end{proof}

\bibliographystyle{plain}
\bibliography{references}

\end{document}